\documentclass[a4paper,11pt,twoside]{amsart}
\usepackage{amsfonts,amsfonts,amssymb}
\usepackage{srcltx} %inverse search
\usepackage{mathrsfs}
\let\mathcal\mathscr
\usepackage[all,ps,cmtip]{xy}
\usepackage[colorlinks=true,citecolor=blue, urlcolor=blue, linkcolor=blue]{hyperref}
\numberwithin{equation}{subsection}

\def\R{{\bf R}}

\def\PA{{\widehat A}}
\def\PAZ{\widehat{A_Z}}
\def\PB{{\widehat B}}
\def\PPB{{\widehat{\widetilde{B}}}}

\def\PK{{\widehat K}}
\def\PY{{\widehat Y}}
\def\PT{{\widehat T}}

\def\phi{{\varphi}}

\newcommand{\abs}[1]{\left\vert#1\right\vert}
\newcommand{\set}[1]{\left\{#1\right\}}

\newcommand{\Db}{{\rm D}^{\rm b}}

\DeclareMathOperator{\rk}{rk}

\def\cI{\mathcal{I}}

\def\cF{\mathcal{F}}

\def\cO{\mathcal{O}}
\def\cG{\mathcal{G}}
\def\cP{\mathcal{P}}

\def\cS{\mathcal{S}}

\def\cQ{\mathcal{Q}}

\def\cW{\mathcal{W}}

\def\ie{\hbox{i.e.,}}

\DeclareMathOperator{\codim}{codim}
\DeclareMathOperator{\Pic}{Pic}

\newtheorem{theAlph}{Theorem}[section]

\newcommand{\res}[2]{\left.#1\right|_{#2}} % restriction bar

\newcommand{\Co}{\hyperref[hypo1]{{\rm (C.1--2)}}} %{{\rm (C.1--2) }}
\newcommand{\Con}{\hyperref[hypo1]{{\rm (C.1--3)}}} %{{\rm (C.1--3) }}
\newcommand{\Cond}{\hyperref[hypo1]{{\rm (C.1--4)}}} %{{\rm (C.1--4) }}

\newcommand{\Vz}{V_{\oldstylenums{0}}}

\newtheorem{theo}{Theorem}[section]
\newtheorem{lemm}[theo]{Lemma}
\newtheorem{coro}[theo]{Corollary}
\newtheorem{prop}[theo]{Proposition}

\newtheorem*{conj*}{Conjecture}

\newtheorem{sett}[theo]{Setting}

\theoremstyle{definition}

\newtheorem{rema}[theo]{Remark}

\newtheorem{exam}[theo]{Example}
\newtheorem{qu}[theo]{Question}
\newtheorem*{qu*}{Question}

\begin{document}
\title[Characterization of products of theta divisors]{Characterization of products of theta divisors}
\author{Zhi Jiang}
\author{Mart\'{\i} Lahoz}
\author{Sofia Tirabassi}

\address{Z.J.: D\'epartement de Math\'ematiques d'Orsay, Universit\'{e} Paris-Sud 11, B\^{a}timent 425, F-91405 Orsay, France}
\email{Zhi.Jiang@math.u-psud.fr}
\urladdr{\url{http://www.math.u-psud.fr/~jiang/}}

\address{M.L.: D\'epartement de Math\'ematiques d'Orsay, Universit\'{e} Paris-Sud 11, B\^{a}timent 425, F-91405 Orsay, France}
\curraddr{Institut de Math\'{e}matiques de Jussieu, Universit\'{e} Paris 7 Denis Diderot, B\^{a}timent Sophie-Germain, Case 7032, F-75205 Paris, France}
\email{lahoz@math.jussieu.fr}
\urladdr{\url{http://www.math.u-psud.fr/~lahoz/}}

\address{S.T.: Math Department, Univesity of Utah, 155 S 1400 E, Salt Lake City, UT 84112, USA}
\email{sofia@math.utah.edu}
\urladdr{\url{http://www.math.utah.edu/~sofia/}}

\thanks{M.~L.~is supported by Fondation Math\'ematique Jacques Hadamard (FMJH) and partially by MTM2012-38122-C03-02.
S.~T.~initiated to work on this project when she was a Postdoc at the Faculty of Mathematics, Informatics and Mechanics of University of Warsaw.
}

\begin{abstract}
	We study products of irreducible theta divisors from two points of view.
	On the one hand, we characterize them as normal subvarieties of abelian varieties such that a desingularization has
	holomorphic Euler characteristic $1$.
	On the other hand, we identify them up to birational equivalence among all varieties of maximal Albanese dimension.
	We also describe the structure of varieties $X$ of maximal Albanese dimension, with holomorphic Euler characteristic $1$ and irregularity $2\dim X-1$.
\end{abstract}
\maketitle

\section*{Introduction}

Given an abelian variety $A$ and an ample line bundle $L$ on $A$, we say that $(A,L)$ is a principally polarized abelian variety if $L$ has only one section up to constants.
The effective divisor associated to $L$ and its translates are called theta divisors.

Irreducible theta divisors of dimension $1$ are smooth genus $2$ curves.
In dimension $2$, they are birational to the symmetric product of a curve of genus $3$.
More precisely, if the genus $3$ curve $C$ is non-hyperelliptic they are isomorphic to $C^{(2)}$, while if the curve is hyperelliptic, they have an isolated singularity and $C^{(2)}$ is their natural desingularization.
In higher dimensions, they may also be singular, but Ein--Lazarsfeld proved in \cite{el} that they are always normal and they have only rational singularities.
Hence, given a theta divisor $\Theta\subset A$, any desingularization $X\to \Theta$ has
\begin{equation*}
	\chi(X,\omega_X)=\chi(\Theta,\omega_\Theta)=h^0(A,\cO_{A}(\Theta))=1.
\end{equation*}
Clearly, $A$ is the Albanese variety of $X$, and the induced morphism $X\to A$ is its Albanese morphism.
Thus, desingularizations of theta divisors are varieties of maximal Albanese dimension, \ie\ varieties $X$ such that the Albanese morphism $a_X:X\to A$ is generically finite onto its image.

\medskip

The characterization up to birational equivalence of irreducible theta divisors among varieties of maximal Albanese dimension started with the work of Hacon \cite{Hac-theta} and it has been further refined in \cite{HP-mAd,LP,BLNP,pareschi} (see the end of Section \ref{sec:mAd} for more precise references).

Naturally, if $X$ is a desingularization of a product of irreducible theta divisors, then we also have $\chi(X,\omega_X)=1$.
Indeed, Beauville (in the appendix to \cite{appBeau}) characterized products of two curves of genus $2$, as smooth projective surfaces of general type $S$ with irregularity $4$ and $\chi(S,\omega_S)=1$.
This characterization of products of curves of genus $2$ (irreducible theta divisors of dimension $1$) was extended to higher dimensional varieties of maximal Albanese dimension by Hacon--Pardini \cite{HP-g2}.

\medskip

We are interested here in giving a more general characterization of products of irreducible theta divisors.
First we identify them inside arbitrary abelian varieties as normal subvarieties such that a desingularization has holomorphic
Euler characteristic $1$.

\begin{theAlph}\label{thmA}
Let $Y$ be a subvariety of an abelian variety $A$.
Let $X$ be a desingularization of $Y$. Then, $Y$ is isomorphic to a product of theta divisors if, and only if, $Y$ is normal and $\chi(X,\omega_X)=1$.
\end{theAlph}

On the one hand, studying subvarieties of abelian varieties is more general than studying just the images of smooth projective varieties under the Albanese morphism.
For example, a singular curve inside an abelian variety is never the image under an Albanese morphism.

On the other hand, given a smooth projective variety, even if it is of maximal Albanese dimension, the Albanese morphism does not need to induce a birational map onto its image.

Indeed, if we consider a subvariety $Y$ of an abelian variety $A$ such that $Y$ is not of general type, then $Y$ is fibred by tori (see \cite[Thm.~10.9]{ueno}).
Hence, in this case, a desingularization $X$ of $Y$ will have $\chi(X,\omega_X)=0$.
Nevertheless, there exist smooth projective varieties $X$ of maximal Albanese dimension such that $\chi(X,\omega_X)=0$ and $X$ is of general type.
Clearly, in this case the Albanese morphism does not induce a birational map.
The first example was constructed by Ein--Lazarsfeld \cite[Ex.~1.13]{el} and Chen--Debarre--Jiang have studied in detail these varieties in \cite{CDJ}.
In particular, it is shown in \cite{CDJ} that the Ein--Lazarsfeld example is essentially the only one in dimension $3$.

In any case, if $X$ is a smooth projective variety of maximal Albanese dimension with $\chi(X,\omega_X)=1$ such that the Albanese image of $X$ is of general type, then the Albanese morphism induces a birational equivalence (see Lemma \ref{lem:sof}).
This allows us to characterize products of irreducible theta divisors up to birational equivalence by using a slightly less restrictive version of Theorem \ref{thmA} (see Theorem \ref{thm:sm-in-codim1}).

\begin{theAlph}\label{thmWC}
Let $X$ be a smooth projective variety.
Then $X$ is birational to a product of irreducible theta divisors if and only if $X$ satisfies the following conditions:
\begin{align}
&X \text{ is of maximal Albanese dimension;}\tag{C.1}\\
&\chi(X, \omega_X)=1;\tag{C.2}\\
&\text{the Albanese image of } X \text{ is of general type;}\tag{C.3}\\
\tag{C.4} &\text{the Albanese image of } X \text{ is smooth in codimension }1.
\end{align}
\end{theAlph}

It is an open question to determine whether condition (C.4) is necessary or not (see Question \ref{qu?}).

Let $X$ be a smooth projective variety $X$ of maximal Albanese dimension and $\chi(X, \omega_X)=1$.
Hacon--Pardini show in \cite{HP-g2} that the irregularity $q(X)=h^0(X,\Omega_X^1)$ sits in the following range:
\begin{equation*}
	\dim X\leq q(X)\leq 2\dim X,
\end{equation*}
and, when $q(X)$ is maximal, \ie\ $q(X)=2\dim X$, $X$ is birational to a product of curves of genus $2$ (theta divisors of dimension $1$).
Hence condition \eqref{hypo4} is not necessary when $q(X)=2\dim X$.
We give a complete classification of the submaximal irregularity case, \ie\ $q(X)=2\dim X-1$ (see Theorem \ref{thm:2n-1main}), which implies that condition \eqref{hypo4} is also not necessary when $q(X)=2\dim X-1$.

\begin{theAlph}\label{thmQ=2D-1}
	Let $X$ be a smooth projective variety of maximal Albanese dimension.
	Assume that $\chi(X,\omega_X)=1$ and $q(X)=2\dim X-1$.
	Then $X$ is birational to one of the following
	varieties:
	\begin{enumerate}
		 \item a product of smooth curves of genus $2$ with the $2$-dimensional symmetric product of a curve of genus $3$;
		 \item $(C_1\times \widetilde{Z})/\langle\tau\rangle$, where
			$C_1$ is a bielliptic curve of genus $2$,
			$\widetilde{Z}\to C_1\times\cdots \times C_{n-1}$ is an \'etale double cover of a product of smooth projective curves of genus $2$, and
			$\tau$ is an involution acting diagonally on $C_1$ and $\widetilde{Z}$ via the involutions corresponding respectively to the double covers.
	 \end{enumerate}
\end{theAlph}

This result generalizes to higher dimensions the classification of surfaces $S$ of general type with $p_g(S)=q(S)=3$ established independently by Hacon--Pardini \cite{HP-surf} and Pirola \cite{pirola}.

% This is part of a more ambitious program: the classification of all smooth projective varieties of maximal Albanese dimension and $\chi(X,\omega_X)=1$.

The classification of the next step, \ie\ $q(X)=2\dim X-2$, seems rather difficult since in this case there does not exist such a nice classification even for surfaces ($\dim X=2$).
Indeed, besides the general example of the double cover of a principally polarized abelian variety $(A,\Theta)$ ramified along a smooth divisor in the linear system $\abs{2\Theta}$, there are other more complicated surfaces satisfying the hypothesis (see \cite{CH2} and \cite{PePo1,PePo2}).

\subsection*{Notation}
Throughout this paper we work over the complex numbers.
More precisely, a variety or a subvariety is a separated scheme of finite type over $\mathbb{C}$ which is integral, \ie\ irreducible and reduced, but not necessarily smooth.
In particular, we will say that a variety is a theta divisor, if it is an irreducible divisor that induces a principal polarization on an abelian variety.

We denote by $a_X: X\to A_X$ the {\em Albanese morphism} of a smooth projective variety $X$.
We say that a variety $X$ is {\em irregular} if $a_X$ is non-trivial, \ie\ $q(X):=h^0(X,\Omega_X^1)\neq 0$.
When $a_X$ is generically finite onto its image, we will say that $X$ is of {\em maximal Albanese dimension}.
Given an abelian variety $A$, we will call $\PA$ its dual, usually interpreted as $\Pic^0A$.
For a morphism $a: X\to A$ to an abelian variety and a coherent sheaf $\cF$ on $X$, we denote by $V^i(\cF, a)$ the {\em $i$-th cohomological support loci}:
\begin{equation*}V^i(\cF, a):=\{P\in\PA\mid H^i(X, \cF\otimes a^*P)\neq 0\}.\end{equation*}
When $a=a_X$, we will often omit it from the notation, \ie\ $V^i(\cF):=V^i(\cF,a_X)$.
We also use $\Vz^i(\cF,a)$ to denote the components of $V^i(\cF, a)$ passing through the identity element of $\PA$.

This paper uses the full machinery of generic vanishing theory, applied to irregular varieties (see \cite{pareschi} for a nice survey).
For example, we will often use that the irreducible components of $V^i(\omega_X,a)$ are abelian subvarieties of $\PB\subseteq \PA$ translated by a torsion point $P\in \PA$, so we will usually denote them by $P+\PB$ (see \cite{GL1,GL2,sim}).
Moreover, if $X$ is of maximal Albanese dimension, then $\codim_\PA V^i(\omega_X,a)\geq i$ for all $i\geq 0$ (see \cite{GL1}).
We will also use the {\em derivative complex} introduced in \cite{GL2}, under the form used in \cite{el}.
When $\cG$ is a sheaf on an abelian variety $A$, we will say that $\cG$ is a {\em GV}-sheaf, if $\codim_\PA V^i(\cG)\geq i$ for all $i\geq 0$ (see \cite[Prop./Def.~2.1]{reg3})
and we will say that $\cG$ is {\em M-regular}, if $\codim_\PA V^i(\cG)>i$ for all $i>0$ (see \cite[Def.~2.1]{pp-reg1}).

For a smooth projective variety $X$, we denote by $\Db(X)$ the bounded derived category of coherent sheaves on $X$.
We refer to \cite{huy} for basic notions on derived categories.

\subsection*{Plan of the paper}
The paper is organized as follows.
Section \ref{sec:prelim} collects some basic facts about morphisms $a:X\to A$ from a smooth projective variety $X$ to an abelian variety $A$.
In particular, we see what happens when $a(X)$ is of general type and $\chi(X,\omega_X)=1$ (Section \ref{ssec:chi=1}).
We also study fibrations where the general fiber is a theta divisor (Section \ref{ssec:fibrTheta}).
Some of these results might be well-known to the experts.

Section \ref{sec:subvA} concerns the case of subvarieties of abelian varieties $Y\subset A$.
We let $X\stackrel{f}\to Y$ be a desingularization of $Y$.
First we review a result of Pareschi (Section \ref{ssec:pareschi}) that allows us to restrict to the case where $V^1(\omega_X,a)$ has positive dimension.
In Section \ref{ssec:divisorcase}, we consider the case where $Y$ has codimension $1$ in $A$.
Finally, in Section \ref{ssec:sub-sm-1}, we consider the general case where $Y$ is smooth in codimension $1$.
The main result of this section is Theorem \ref{thm:sm-in-codim1-sav} and we deduce Theorem \ref{thmA} from it.
Theorem \ref{thm:sm-in-codim1-sav} is proven by induction on the number of irreducible components of $\Vz^1(\omega_X,a_X)$.
Hence, the irreducible components of $\Vz^1(\omega_X,a_X)$ and their relation with the irreducible components of $V^i(\omega_X,a)$ for $i\geq 1$, play a major role in Section \ref{sec:subvA}.

Section \ref{sec:mAd} concerns the case where $X$ is a smooth projective variety of maximal Albanese dimension with $\chi(X,\omega_X)=1$.
First, we classify these varieties when they have submaximal irregularity.
We use Fourier-Mukai techniques to prove Proposition \ref{prop:hpTech}, which is the main technical tool to prove the main Theorem \ref{thm:2n-1main} of this section (which corresponds to Theorem \ref{thmQ=2D-1}).
Finally, we recall a question of Pareschi (see Question \ref{qu?}) and we deduce Theorem \ref{thmWC} (see Theorem \ref{thm:sm-in-codim1}) from Theorem \ref{thm:sm-in-codim1-sav}.
% We also give an easier proof of Theorem \ref{thmWC} by using induction on the dimension of $X$, instead of induction on the number of irreducible components of $\Vz^1(\omega_X,a_X)$.

\subsection*{Acknowledgements}
It is a pleasure to thank Miguel \'Angel Barja, Olivier Debarre, Christopher Hacon, Joan Carles Naranjo, Rita Pardini, and Gian Pietro Pirola for very useful conversations and comments.

Special thanks are due to Giuseppe Pareschi who informed us that he has also proven Theorem \ref{thmWC} under the hypothesis that $a_X(X)$ is normal.
We are very grateful to him for sending us a preliminary version of his paper.
We would also like to thank the unknown referee for suggesting improvements to the exposition.

\bigskip

\section{Preliminaries}\label{sec:prelim}

The following list of equivalences is certainly well-known to the experts.
We state it for easy reference since it will be needed repeatedly hereafter.
\begin{lemm}\label{lem:hypo3}
	Let $a:X\to A$ be a generically finite morphism from $X$ a smooth projective variety to $Y:=a(X)$ a subvariety of an abelian variety $A$.
	Then the following conditions are equivalent:
	\begin{enumerate}
	\item $Y$ is not fibred by positive dimensional abelian varieties;
	\item $Y$ is of general type;
	\item $a_{*}\omega_X$ is a M-regular sheaf;
	\item for any $f: V\to X$ surjective morphism where $V$ is smooth projective, any direct summand of $a_{*}R^jf_*\omega_V$ is M-regular for any $j\geq0$;
	\end{enumerate}
Moreover, the previous equivalent conditions imply that $\chi(X,\omega_X)>0$.
\end{lemm}
\begin{proof}
The equivalence of {\it (i)} and {\it (ii)} is due to Ueno (see \cite[Thm.~10.9]{ueno} or \cite[Thm.~3.7]{mor}).

We now show that {\it (i)} is equivalent to {\it (iii)}.

If {\it (i)} holds and $a_{*}\omega_X$ is not M-regular, then there exists an irreducible component $T$ of $V^i(a_{*}\omega_X)$ of codimension $i$ in $\PA$ for some $i\geq 1$, by generic vanishing (e.g.,~\cite[Thm.~2.2]{HP-g2}).
Then as in \cite[Proof of Thm.~3, p.~249]{el} and using \cite[\S4]{GL2}, we get that $Y$ is fibred by the fibres of the natural fibration $A\to \PT$, which is a contradiction (see also \cite[Rmk.~1.6]{el}).

Now assume that {\it (iii)} holds.
Let $K$ be the maximal abelian subvariety of $A$ such that $K+Y=Y$.
Let $j=\dim K$.
By \cite[Thm.~3]{el}, the image $Z$ of the projection $Y\to A/K$ is of general type and given any desingularization $\mu: Z'\rightarrow Z$,
we have $\chi(Z', \omega_{Z'})>0$.
Consider an appropriate birational model of the Stein factorization of the morphism $X\rightarrow Y \rightarrow Z$, such that there exists $Z''$ a smooth projective variety that admits a generically finite and surjective morphism $g:Z''\rightarrow Z'$ and a fibration $h:X\rightarrow Z''$.
Observe that $\chi(Z'',\omega_{Z''})\geq \chi(Z',\omega_{Z'}) >0$.
We have $\omega_{Z''}\simeq R^jh_*\omega_X$ (\cite[Prop.~7.6]{ko1}).
Since $Z''$ has positive Euler characteristic, we have that
$$\widehat{A/K} = V^0(\omega_{Z''},\mu\circ g)=V^0(R^j h_*\omega_X,\mu\circ g).$$
Thus, Koll\'{a}r's theorem \cite[Thm.~3.1]{ko2}, or the refined version by Hacon--Pardini \cite[Thm.~2.1]{HP-g2}, implies that the codimension-$j$ abelian subvariety $\widehat{A/K}\subset \PA$ is an irreducible component of $V^j(a_{*}\omega_X)$.
Therefore, $j=0$ and $Y$ is not fibred by positive dimensional abelian varieties.

On the one hand, condition {\it (iv)} clearly implies {\it (iii)}.
On the other hand, if $a_{*}R^jf_*\omega_V$ is not M-regular, then there exists an irreducible component $T$ of $V^i(a_{*}R^jf_*\omega_V)$ of codimension $i$ in $\PA$ for some $i\geq 1$, by \cite[Thm.~2.2]{HP-g2}.
As before, the same argument as in \cite[Proof of Thm.~3, p.~249]{el} using \cite[\S4]{GL2}, yields to $Y$ being fibred by the fibres of the natural fibration $A\to \PT$, which is a contradiction.
Since a direct summand of a M-regular sheaf is M-regular, we are done.

In any case, since $X$ is of maximal Albanese dimension, then $\chi(X,\omega_X)\geq 0$ (see \cite[Cor., p.~390]{GL1}).
Moreover, by \cite[Thm.~3]{el}, {\it (i)} implies $\chi(X,\omega_X)>0$.
\end{proof}

\begin{rema}$\quad$\label{rem:chi>0-Xgt}
\begin{enumerate}
	\item On the one hand, in the notation of Lemma \ref{lem:hypo3}, if $\chi(X,\omega_X)>0$, then $X$ is of general type.
	This is certainly not the converse to the last statement of Lemma \ref{lem:hypo3}, since $X$ of general type does not imply that $Y$ is of general type (see \cite[Ex.~1.13]{el}).

	The fact that $\chi(X,\omega_X)>0$ implies that $X$ is of general type follows from \cite{Campana} or \cite[Cor.~2.4]{CH1}\footnote{There is a small misprint in the statement of this corollary and $q(X)$ should be $\dim A$.}.

	\item On the other hand, the fact that the equivalent conditions imply $\chi(A,a_*\omega_X)=\chi(X,\omega_X)>0$ can be also deduced via a more general argument from the fact that $a_*\omega_X$ is a M-regular sheaf. \\
	Indeed, we recall that if $\cG$ is a non-trivial M-regular sheaf on $A$, then $V^0(\cG)=\PA$ (see for instance \cite[Lem~1.12]{pareschi}) and, by the invariance of the Euler characteristic, $\chi(A,\cG)>0$.
%GV-sheaf on an abelian variety, then $\chi(A, \cG)\geq 0$ and, moreover, if $\cG$
\end{enumerate}
\end{rema}

\medskip

From now on, we will always have $a:X\to A$ a generically finite morphism from $X$ a smooth projective variety to $Y:=a(X)$ a subvariety of an abelian variety $A$.
For simplicity we will assume that no translate of $Y$ is contained in any abelian subvariety of $A$, \ie\ $Y$ generates $A$.
We will view this setting from two different angles.
\begin{enumerate}
	\item The first point of view considers $Y$ given and let $f:X\to Y$ be a desingularization.
		Thus, $f:X\to Y$ will always be a birational morphism, but $a:X\to A$ is not necessarily the Albanese morphism of $X$.\\
		This will be treated in Section \ref{sec:subvA}.
	\item The second point of view considers $X$ a given variety of maximal Albanese dimension and let $a:X\to A$ be the Albanese morphism, \ie\ $a_X:X\to A_X$.
		In this case, $f:X\to Y$ is not necessarily birational, since it may have degree greater than 1. \\
		This will be treated mainly in Section \ref{sec:mAd}.
\end{enumerate}
Note that, in both cases $X$ is of maximal Albanese dimension.
Clearly, the study of the two problems is very interrelated.

Recall that by Lemma \ref{lem:hypo3} and Remark \ref{rem:chi>0-Xgt}{\it(i)}, we have that $\chi(X,\omega_X)=0$ implies that $Y$ is not of general type, \ie\ $Y$ is fibred by tori.
So, from both perspectives, we will be mainly interested in the``next'' case, \ie\ when $\chi(X,\omega_X)=1$.

\subsection{The hypothesis $\chi(X,\omega_X)=1$}\label{ssec:chi=1}
Let $a:X\to A$ be a generically finite morphism from $X$ a smooth projective variety to $Y:=a(X)$ a subvariety of an abelian variety $A$.
Before trying to study our problem from the two different perspectives, we give some general results.
\begin{lemm}[{\cite[Thm.~5.2.2]{Ti}}]\label{lem:sof}
	Let $a:X\to A$ be a generically finite morphism from $X$ a smooth projective variety to $Y:=a(X)$ a subvariety of an abelian variety $A$.

	If $Y$ is of general type and $\chi(X,\omega_X)=1$, then $a: X\to A$ induces a birational equivalence between $X$ and $Y$.
\end{lemm}
\begin{proof}
	We can consider a birational model of $f:X\to Y$ such that $X$ and $Y$ are smooth and we denote by $g:Y\to A$ the morphism that is birational onto its image.
	Then, since $f$ is generically finite, the trace map $f_*\omega_X\twoheadrightarrow \omega_Y $ splits the natural inclusion $\omega_Y\rightarrow f_*\omega_X$.
	Hence $f_*\omega_X\simeq \omega_Y\oplus \cQ$, where $\cQ$ is the kernel of the trace map, and $a_*\omega_X\simeq g_*\omega_Y\oplus g_*\cQ$.
	By Lemma \ref{lem:hypo3}, $g_*\omega_{Y}$ and $a_*\omega_{X}$ are M-regular, so, by Remark \ref{rem:chi>0-Xgt}{\it(ii)}, either $\cQ$ is trivial and $f$ is birational or $g_*\cQ$ is M-regular and $\chi(A,g_*\cQ)>0$.

	By \cite[Thm.~3.4]{ko2}, $\chi(A,a_*\omega_X)=\chi(X,\omega_X)=1$ and $\chi(A,g_*\omega_{Y})=\chi(Y,\omega_Y)$.
	Since, by hypothesis $Y$ is of general type, Lemma \ref{lem:hypo3} implies that $g_*\omega_{Y}$ is M-regular, so $\chi(A,g_*\omega_{Y})>0$.
	Thus, $\chi(A, g_*\cQ)=0$.
	Hence $\cQ=0$ and $f$ is birational.
\end{proof}

Combining Lemma \ref{lem:hypo3} with \cite[Lem.~3.2]{HP-g2}, we have
\begin{lemm}\label{lem:V^1}
	Let $a:X\to A$ be a generically finite morphism from $X$ a smooth projective variety to $Y:=a(X)$ a subvariety of an abelian variety $A$.
	Suppose $Y$ is of general type and $\chi(X,\omega_X)=1$.

	If $T$ is any irreducible component of $V^1(\omega_X,a)$ of $\codim_{\PA}T=i+1$, then $T$ is an irreducible component of $V^i(\omega_X,a)$.
\end{lemm}

\begin{proof}
By \cite[Lem.~3.2]{HP-g2}, we know that $T\subset V^i(\omega_X,a)$.
Since $a_{*}\omega_X$ is M-regular, $T$ is an irreducible component of $V^i(a_*\omega_X)=V^i(\omega_X,a)$.
\end{proof}

\subsection{Fibrations in theta divisors}\label{ssec:fibrTheta}
Under the hypothesis $\chi(X,\omega_X)=1$, the main aim will be to construct a fibration from $Y$ to some other variety such that the fibres are theta divisors of a fixed abelian subvariety of $A$.
Fibrations by theta divisors are very special.
In particular, they are isotrivial, so very close to being a product.
Indeed, the following proposition shows that they are products under some mild extra assumptions (compare with \cite[Thm.~5.2]{zhang}).
The idea of the proof goes back to \cite{serrano} and it was already used in \cite[proof of Thm.~3.1]{HP-g2}.

\begin{prop}\label{prop:fibertheta}
Let $f: X\to Y$ be a fibration between smooth projective varieties of maximal Albanese dimension.
Assume that a general fiber $X_y$ is birational to a theta divisor $\Theta$ and $q(X)-q(Y)>\dim X-\dim Y$.
Then $X$ is birational to $\Theta\times Y$.
\end{prop}
\begin{proof}
We denote by $K$ the kernel of the natural quotient $f_*: A_X\to A_Y$. Then $\res{a_X}{X_y}: X_y\rightarrow K_y$ is a generically finite morphism, where $K_y$ is the corresponding translate of $K$ in $A_X$.
From the universal property of Albanese morphisms, we know that the induced morphism $A_{X_y}\to K_y$ is surjective.
By hypothesis, $\dim K=q(Y)-q(X) \geq \dim X_y+1=\dim A_{X_y}$.
Thus, $A_{X_y}$ is isogeny onto $K$, so $A_{X_y}$ is fixed, which implies that $f$ is isotrivial.

Note that, since $\Theta$ has canonical singularities, the canonical model of $X_y$ is fixed and isomorphic to $\Theta$.
Thus, the fibers of the relative canonical model $X'\stackrel{f}\to Y$ are isomorphic to $\Theta$ in an open Zariski subset of $Y$.
Since $\Theta$ is of general type, $\mathrm{Aut}(\Theta)$ is finite.
Therefore, there exists a  finite Galois cover $\PY\to Y$ with group $G$ such that the main component of $X'\times_Y\PY$ is birational to $\Theta\times \PY$.
 So we can consider the commutative diagram
\begin{equation*}
	\xymatrix{
	\Theta\times \PY \ar[r]_{/G}\ar[d] &X'\ar[d]^{f'}\\
	\PY\ar[r]_{/G} &Y,
	}
\end{equation*}
where $G$ acts on $\Theta\times \PY$ diagonally and the horizontal morphisms are quotients by $G$.

Now, we consider $V:=\Theta/G$ and the following commutative diagram
\begin{equation*}
\xymatrix@!C{&W \ar[r]^{\varphi} & V \ar[r]^{a_V}& A_V\\
W\times \PY \ar[ur]^{p_1}\ar[r]\ar[d]_{p_2} & {X'}\ar[ur]^g\ar[r]^{a_{X'}}\ar[d]_{f'}& A_{X'}\ar[d]\ar[ur]\\
\PY\ar[r] & Y\ar[r]^{a_Y}& A_Y,}
\end{equation*}
where $W$ is birational to $\Theta$ and after birational modifications and abusing notation, we will assume that all varieties are smooth projective.
We want to prove that $\varphi$ is birational.

We denote by $K$ the kernel of $A_{X'}\to A_Y$.
Note that, $h^1(X',\cO_{X'})= (h^1(\PY,\cO_{\PY})\oplus h^1(W,\cO_W))^G$.
Hence, since $q(X)-q(Y)>\dim X-\dim Y$, then $A_{V}$ is isogenous to $K$.
In particular, $\dim K=\dim A_V=\dim \Theta+1$.

Since $W$ is birational to a theta divisor $\Theta$, for any $P\in \PA_V$ such that $\varphi^*t^*P$ is not trivial, we have $h^0(W, \omega_{W}\otimes \varphi^*t^*P)=1$ and
\begin{align*}&h^i(W, \omega_{W}\otimes \varphi^*t^*P)=0 &\text{ for all }i\geq 1.
\end{align*}
By Grauert--Riemenschneider vanishing, $R^i\varphi_*\omega_W=0$ for all $i>0$ and
\begin{equation*}\varphi_*\omega_{W}=\omega_V\oplus \cQ,\end{equation*}
where $\cQ$ is a torsion-free sheaf.

Since $\Theta$ has only rational singularities by \cite[Thm.~1]{el}, $\dim V^i(\omega_W,t\circ \varphi)=0$, so we have also that $\dim V^1(\omega_V, t)=0$, using that  $R^i\varphi_*\omega_W=0$ for all $i>0$.
Hence, $t_*\omega_V$ is M-regular and then $V^0(\omega_V, t)=\Pic^0(A_{X'}/B_2)$.
Thus, for any $P\in \Pic^0(A_{X'}/B_2)$ such that $\varphi^*t^*P$ is not trivial, we also have $h^0(V, \omega_V\otimes t^*P)=1$.
Then $t_*\cQ$ is also M-regular with $\dim V^0(t_*\cQ)=0$, so $t_*\cQ=0$ (recall Remark \ref{rem:chi>0-Xgt}{\it(ii)}) and $\varphi$ is birational.
Therefore, $(g, f): {X'}\to V\times Y$ is birational.
\end{proof}

\begin{rema}
	The condition $q(X)-q(Y) > \dim X -\dim Y$ is necessary.

	Indeed, let $C\to E$ be a bielliptic curve of genus $2$.
	Consider $\widetilde{C}\to C$ an \'etale double cover and $X=(C\times \widetilde{C})/\langle\tau\rangle$, where $\tau$ is an involution acting diagonally on $C$ and $\widetilde{C}$ via the involutions corresponding respectively to the double covers.
	
	Then, $X$ has a natural fibration $X\to C$ such that the general fiber is $C$.
	Note that $q(X)=3$ and $q(C)=2$ and, indeed, $X$ is not birational to a product of curves of genus $2$.
	Varieties in Theorem \ref{thmQ=2D-1}{\it (ii)} generalize this construction to higher dimensions.
\end{rema}

\bigskip

\section{Subvarieties of abelian varieties with holomorphic Euler characteristic 1} \label{sec:subvA}

Let $Y\subseteq A$ be subvariety of an abelian variety $A$.
Without loss of generality, we assume through \emph{all} Section \ref{sec:subvA} that $Y$ generates $A$ as a group.
We also assume that $Y$ is reduced and irreducible but not necessarily smooth.
We consider $X\stackrel{f}\to Y\subseteq A$ a desingularization of $Y$.
We want to study varieties $Y$ such that $\chi(X,\omega_X)=1$.

\subsection{Subvarieties with $\chi(X,\omega_X)=1$ and small $V^i(\omega_X,a)$}\label{ssec:pareschi}

We start by reformulating a theorem of Pareschi (see {\cite[Thm.~5.1]{pareschi}}).

\begin{theo}\label{thm:pareschi}
	Let $X\stackrel{f}\to Y$ be a desingularization of a subvariety $Y\subseteq A$ of an abelian variety and denote $a: X\xrightarrow{f} Y\hookrightarrow A$ the induced morphism to $A$.
	Assume that $\chi(X,\omega_X)=1$.
	If $V^i(\omega_X,a)$ has no components of codimension $i+1$ for all $0<i<\dim X$, then $Y$ is birational to a theta divisor.
	Moreover, the induced morphism $t: A_X\to A$ is an isogeny.
\end{theo}

Pareschi proved the above theorem using Fourier-Mukai techniques and homological algebra.
We give an alternative proof, which is closer to the spirit of this paper.
The ingredients of our proof are Lemma \ref{lem:V^1} and Hacon's characterization of theta divisors (see \cite[Sect.~3]{Hac-theta}).

\begin{proof}
Denote by $n=\dim Y=\dim X$.

We consider an irreducible component $T$ of $V^1(\omega_X,a)=\bigcup_{k\geq 1}V^k(\omega_X,a)$.
Denote by $j+1$ the codimension of $T$ in $\PA$.
By Lemma \ref{lem:V^1}, we know that $T$ is an irreducible component of $V^j(\omega_X,a)$.
By assumption, $j=n$ and $T$ contains $0$ as an isolated point, since $Y$ generates $A$, \ie\ $\Pic^0A\to\Pic^0X$ is an isogeny onto its image.
We then conclude that $\dim A=j+1=n+1$ and $V^1(\omega_X,a)$ contains $0$ as an isolated point.

Now we consider the derivative complex over $\mathbb{P}(H^1(A, \cO_A))=\mathbb{P}^{n}$ as in \cite[Proof of Thm.~3]{el},
\begin{equation*}H^0(X, \omega_X)\otimes \cO_{\mathbb{P}^{n}}(-n)\to\cdots \to H^n(X, \omega_X)\otimes \cO_{\mathbb{P}^{n}}\to 0.\end{equation*}
Since $0$ is an isolated point in $V^1(\omega_X,a)$ and $\chi(X,\omega_X)=1$, the above complex is exact and the kernel on the left is a line bundle $\cO_{\mathbb{P}^{n}}(-t)$.
By chasing through the cohomology of the above complex, we see easily $t=n+1$ and $h^i(X, \omega_X)=\binom{n+1}{i+1}$.
Thus, Hacon's criterion \cite[Corollary 3.4]{Hac-theta} implies that $X$ is birational to a theta divisor.
Hence, so is $Y$.
Since $q(X)=n+1$, it is clear that $t: A_X\to A$ is an isogeny.
\end{proof}

\begin{rema}\label{rem:notppav_divcase}
	Note that $A$ does not need to be principally polarized.
	Indeed, given $\Theta\subset A_{\Theta}$ we can consider any isogeny $A_{\Theta}\to A$ and let $Y\subset A$ the image of $\Theta$.
\end{rema}

The previous theorem exhibits the importance of the codimension-$(i+1)$ irreducible components of $V^i(\omega_X,a)$.
Thus, we place ourselves in the following setting fixing the notation.

\begin{sett}\label{set:c1}
	Let $a:X\to A$ be a generically finite morphism from $X$ a smooth projective variety to $Y:=a(X)$ a subvariety of an abelian variety $A$.
	Suppose that there exists a codimension-$(i+1)$ irreducible component of $V^i(\omega_X,a)$ for some $1\leq i\leq \dim X-1$.

	We can describe this component as $P+\PB$, where $\PB$ is an abelian subvariety of $\PA$ and $P$ is a torsion point (see \cite[Thm.~0.1]{GL2} and \cite[\S4,5,6]{sim}).

	Consider the composition of the $a:X\stackrel{f}\to Y\hookrightarrow A$ and the quotient $\pi_B:A\to B$ given by dualizing the inclusion $\PB\hookrightarrow \PA$.
	Then, denote by $Y_B:=\pi_B(Y)$ and let $X\stackrel{\beta}\to X_B \stackrel{f_B}\to Y_B$ be Stein factorization of $X\to Y_B$.
	Denote by $X_b$ a general fiber of the fibration $\beta$ and denote by $Y_b:=f(X_b)$ and $f_b:=\res{f}{X_b}$.

	Varieties and morphisms described above fit in the following commutative diagram:
	\begin{equation}\label{eqn:xy1}
	\xymatrix{
	X_b\ar@{^{(}->}[d]\ar[r]^{f_b}& Y_b \ar@{^{(}->}[d] \\
	X\ar[r]^f\ar[d]_{\beta} & Y\ar[d]\ar@{^{(}->}[r] & A\ar[d]^{\pi_B}\\
	X_B\ar@/_1.5pc/[rr]_{a_{B}}\ar[r]^{f_B}& Y_B\ar@{^{(}->}[r] & B.}
	\end{equation}
	After birational modifications and abusing notation, we may assume that all varieties in the leftmost column are smooth and projective.
\end{sett}

We summarize some useful properties in the following lemma.
\begin{lemm}[Properties of Setting \ref{set:c1}] \label{lem:propsett1}
With the same notations as in the previous setting.
\begin{enumerate}
\item The morphism $f_B$ is generically finite and the image $Y_b:=f(X_b)$ is an irreducible component of the corresponding fiber of $Y\to B$.
\item If $f$ is birational and $Y$ is of general type, then the morphism $f_b: X_b\to Y_b$ is birational. Moreover, $Y_b$ is an irreducible ample divisor in a translate of the kernel of $\pi_B$.
\end{enumerate}
\end{lemm}
\begin{proof}
The first assertion is clear from the description.
Since $f$ is birational, $f_b$ is also birational.
Since $P+\PB$ is an irreducible component of $V^i(\omega_X,a)$, by \cite[Thm.~0.2]{GL2}, $\dim X-\dim X_B\geq i$.
Moreover, $Y$ is not fibred by abelian subvarieties (see Lemma \ref{lem:hypo3}), so $\dim X_B=\dim X-i$ and $Y_b$ is an ample divisor in a translate of the kernel of $\pi_B$.
\end{proof}

\bigskip

\subsection{Divisors in abelian varieties}\label{ssec:divisorcase}
If we assume that $Y$ is an ample divisor (in particular, conditions in Lemma \ref{lem:hypo3} are satisfied), then we show that $V^i(\omega_X,a)$ are ``small'' when $Y$ is smooth in codimension $1$.

\begin{lemm}\label{lem:divisorcase}
	Let $X\stackrel{f}\to Y$ be a desingularization of an irreducible ample divisor $Y\subseteq A$ of an abelian variety and denote $a: X\xrightarrow{f} Y\hookrightarrow A$.
	If $V^i(\omega_X,a)$ has a component $T=P+\PB$ of codimension $i+1$ in $\PA$ for some $0<i<\dim X$, then $Y$ is not smooth in codimension $1$.
\end{lemm}
The proof is contained in the proof of \cite[Lem.~6]{DH}.
For reader's convenience, we present it here.

\begin{proof}
By \cite[Prop.~3.1]{el}, there exists an adjoint ideal $\cI_Z\subset \cO_A$, co-supported at the singular locus of $Y$, which sits in the following exact sequence
\begin{equation*}0\to \cO_A\to \cI_Z\otimes\cO_A(Y)\to a_*\omega_X\to 0.\end{equation*}

Let $p: A\to B$ be the natural projection.
We then have the long exact sequence
\begin{equation*} R^{i}p_*P\to R^ip_*(\cI_Z\otimes\cO_A(Y))\to R^ip_*(a_*\omega_X\otimes P)\xrightarrow{\delta} R^{i+1}p_*P\end{equation*}

Since $P+\PB\subset V^i(\omega_X,a)$, the coherent sheaf $R^ip_*(a_*\omega_X\otimes P)$ is non-trivial and $V^0(R^ip_*(a_*\omega_X\otimes P))=\PB$.
Moreover, Koll\'{a}r's theorem \cite[Thm.~3.4]{ko2} implies that $R^ip_*(a_*\omega_X\otimes P)$ is torsion-free on $B$.
Since $R^{i+1}p_*P$ is a flat line bundle on $B$, the boundary map $\delta$ is not injective.
Thus the support of $R^ip_*(\cI_Z\otimes\cO_A(Y))$ is $B$.

Now, we consider the short exact sequence
\begin{equation*}0\to \cI_Z\otimes\cO_A(Y)\to \cO_A(Y)\to \cO_Z(Y)\to 0.\end{equation*}
Since $R^{i}p_*(\cO_A(Y)\otimes P)=0$, from the long exact sequence
\begin{equation*}R^{i-1}p_*(\cO_Z(Y)\otimes P)\to R^ip_*(\cI_Z\otimes\cO_A(Y)\otimes P)\to R^{i}p_*(\cO_A(Y)\otimes P)\end{equation*}
and the previous discussion, we get that the support of $R^{i-1}p_*(\cO_Z(Y)\otimes P)$ is $B$.
Thus, there exists an irreducible component $Z'$ of $Z$ such that $p|_{Z'}: Z'\to B$ is surjective and a general fiber of $p|_{Z'}$ is of dimension $\geq i-1$.
As a consequence, $\dim Z\geq \dim B+i-1=\dim X-1$.

Since $Z$ is contained in the singular locus of $Y$, we conclude that $Y$ is not smooth in codimension $1$.
\end{proof}

As we have seen, the hypothesis that $Y$ is smooth in codimension $1$ arises naturally when $Y$ is a divisor and we want to control the size of $V^i(\omega_X,a)$.
Moreover, as we will see in the following lemma, it implies that $A_X\to A$ is a {\em primitive morphism} of abelian varieties.

\begin{lemm}\label{lem:primitive}
	Let $X\stackrel{f}\to Y$ be a desingularization of a subvariety $Y\subseteq A$ of an abelian variety and denote by $a: X\xrightarrow{f} Y\hookrightarrow A$ the induced morphism to $A$.
	Assume that $Y$ is a smooth in codimension $1$ ample divisor of $A$.
	Then, given $P\in \PA$, $a^*P\cong \cO_X$ implies $P=\cO_A$.
\end{lemm}

\begin{proof}
	We argue by contradiction.
	Suppose that there exists a non-trivial $P\in \PA$ such that $a^*P\cong \cO_X$.
	Then $a: X\to A$ factors through a non-trivial \'etale cover $\pi: \widetilde{A}\to A$ induced by $P$ and $f$ factors through $X\to \pi^{-1}(Y)\to Y$.
	Since $Y$ is smooth in codimension $1$ and it is an ample divisor of $Y$, $\pi^{-1}(Y)$ is irreducible.
	This is a contradiction, since $f$ is birational, but $\pi^{-1}(Y)\to Y$ is not.
\end{proof}

The following corollary will be the ``divisor case'' in the next section and it will play the role of ``building brick''.
It is obtained by combining Theorem \ref{thm:pareschi} and Lemmas \ref{lem:divisorcase} and \ref{lem:primitive}.

\begin{coro}\label{cor:divisorcase}
	Let $Y$ be a smooth in codimension 1 divisor of an abelian variety $A$.
	Let $X$ be a desingularization of $Y$. Then, $Y$ is a theta divisor of $A$ if, and only if, $\chi(X,\omega_X)=1$.
\end{coro}
\begin{proof}
	Ein--Lazarsfeld prove in \cite[Thm.~1]{el} that irreducible theta divisors are normal and have only rational singularities.
	Hence, if $Y$ is isomorphic to a theta divisor, in particular, it is smooth in codimension $1$.

	Now assume that $\chi(X,\omega_X)=1$ and note that $X$ is of general type by Remark \ref{rem:chi>0-Xgt}{\it(i)}, and so is also $Y$.
	By Lemma \ref{lem:divisorcase}, for all $0<i<\dim X$, $V^i(\omega_X,a)$ has no component of codimension $i+1$.
	Therefore, Theorem \ref{thm:pareschi} implies that $Y$ is birational to a theta divisor $\Theta$ of $A_X$ and $t: A_X\to A$ is an isogeny.
 By Lemma \ref{lem:primitive}, we have $\deg t=1$.
\end{proof}

\bigskip

\subsection{Subvarieties smooth in codimension 1}\label{ssec:sub-sm-1}
In this section, we identify products of theta divisors among all normal subvarieties of a given abelian variety (see Theorem \ref{thmA}).
We will deduce this characterization from the following theorem.

\begin{theo}\label{thm:sm-in-codim1-sav}
Let $Y$ be a smooth in codimension 1 subvariety of an abelian variety $A$.
Let $X$ be a desingularization of $Y$. Then, $Y$ is birational to a product of theta divisors if, and only if, $\chi(X,\omega_X)=1$.
 \end{theo}

The rest of this section is mainly devoted to the proof of this theorem.

By Theorem \ref{thm:pareschi}, we can reduce our study to the case where there exists a codimension-$(i+1)$ irreducible component of $V^i(\omega_X,a)$ for some $1\leq i\leq \dim X-1$.
That is, we can assume we are in Setting \ref{set:c1} and use the notation therein.

The following technical lemma shows that, under the condition that $Y$ is smooth in codimension $1$, we may assume in Setting \ref{set:c1} that $X_B$ and $Y_B$ are of general type and, in particular $P=0$, \ie\ the codimension-$(i+1)$ irreducible component passes through zero.

\begin{lemm}\label{lem:pass0}
	Let $A$ be an abelian variety.
	Let $X\stackrel{f}\to Y$ be a desingularization of a subvariety $Y\subseteq A$ smooth in codimension 1 and denote by $a: X\xrightarrow{f} Y\hookrightarrow A$ the induced morphism to $A$.
	Assume that $\chi(X,\omega_X)=1$.

	For any $i>0$, all codimension-$(i+1)$ irreducible components of $V^i(\omega_X,a)$ pass through zero.
	Moreover, using the notation in Setting \ref{set:c1}, if $\PB$ is an irreducible component of $V^i(\omega_X,a)$, then $X_B$ and $Y_B$ are birational and of general type.
\end{lemm}

\begin{rema}\label{rem:V1pass0}
	Since $V^1(\omega_X,a)$ is non-empty, we can suppose that there always exists $P+\PB$ a codimension-$(i+1)$ irreducible component of $V^i(\omega_X,a)$ by Lemma \ref{lem:V^1}.

	In particular, we prove that all irreducible components of $V^1(\omega_X,a)$ pass through zero and the corresponding $X_B$ and $Y_B$ are of general type.
\end{rema}

\begin{proof}
	Let $P+\PB$ a codimension-$(i+1)$ irreducible component of $V^i(\omega_X,a)$.

	If $i=\dim X$, then $\dim B=0$ and $Y$ is an ample divisor of $A$.
	Then, we conclude by Lemma \ref{lem:primitive}.
	Hence, we assume now that $i<\dim X$ and we are in Setting \ref{set:c1} and we use the notation therein.

	\smallskip
	{\bf Step 1.} $f_B:X_B\to Y_B$ is birational.

	Assume by contradiction that $\deg f_B>1$.
	Then, consider two distinct preimages $b_1, b_2\in f_B^{-1}(b)$ and their corresponding fibres $X_{b_i}=\beta^{-1}(b_i)$.
	Recall that $\chi(X,\omega_X)>0$ implies that $X$ is of general type (see Remark \ref{rem:chi>0-Xgt}{\it(i)}), and so is $Y$.
	Then, by Lemma \ref{lem:propsett1}{\it (ii)}, $X_{b_i}$ are mapped birationally to $Y$ via $f$.
	Since $Y_{b_i}:=f(X_{b_i})$ are two different ample divisors in a translate of the kernel of $\pi_B$ (see again Lemma \ref{lem:propsett1}), their intersection has codimension $1$ in each divisor.
	This is in contradiction with $Y$ being smooth in codimension $1$.

	\smallskip

	{\bf Step 2.} $P\in \PB$, so we may assume $P=0$.

	By \cite[Thm.~2.2]{HP-g2}, $R^j\beta_*(\omega_X\otimes a^*P)$ are $GV$-sheaves for all $j\geq 0$.
	Then by \cite[Thm. 3.1]{ko2}, for $P_B\in \PB$ general $h^i(X,\omega_X\otimes a^*P \otimes a^*P_B)= h^0(Y, R^i\beta_*(\omega_X\otimes a^*P)\otimes a_{B}^*P_B)$.
	Since $P+\PB\subset V^i(\omega_X, a)$, $R^if_*(\omega_X\otimes a^*P)\neq 0$.
	This implies that $\res{(a^*P)}{X_b}\cong \cO_{X_b}$.

	Hence, if we consider $a_b: X_b\xrightarrow{f_b} Y_b\hookrightarrow K_b$, then $a_b^*(\res{P}{K_b})\cong \cO_{X_b}$.
	Since $Y_b$ is smooth in codimension $1$ and it is an ample divisor of $K_b$, then $\res{P}{K_b}\cong\cO_{K_b}$ by Lemma \ref{lem:primitive}.
	Since we have the exact sequence of abelian varieties
 \begin{equation*}
		0\to \PB\to\PA\to\PK\to 0,
 \end{equation*}
	$\res{P}{K_b}\cong\cO_{K_b}$ implies that $P\in\PB$.

	\smallskip

	{\bf Step 3.} $X_B$ and $Y_B$ are of general type.

	By Step 1, we only need to prove that $X_B$ is of general type.

	For $P_B\in \PB\setminus\bigcup_j V^1(R^j\beta_* \omega_X, a_{B})\neq \emptyset$ (e.g.~\cite[Thm.~2.2]{HP-g2}),
	\begin{align*}
	0&< h^i(X,\omega_X\otimes a^*P_B)=h^0(Y,R^i \beta_*\omega_X\otimes a_{B}^*P_B) \\
	&=h^0(X_B,\omega_{X_B} \otimes a_{B}^*P_B) & \text{{\cite[Prop.~7.6]{ko1}}}.
	\end{align*}
	By generic vanishing, $\chi(X_B,\omega_{X_B})=h^0(X_B,\omega_{X_B} \otimes a_{B}^*P_B)$ for a general $P_B\in \PB$.
	Hence $\chi({X_B},\omega_{X_B})>0$ and we conclude by Remark \ref{rem:chi>0-Xgt}{\it(i)}.
\end{proof}

% \begin{rema}\label{rem:careful}
% 	We would like to emphasize that we have used the following two arguments:
% 	\begin{enumerate}
% 		\item if $Y_B$ is of general type, then $X_B$ is also of general type, which is equivalent to $P\in \PB$;
% 		\item any irreducible component of $V^1(\omega_X,a)$ of codimension $i+1$ is an irreducible component of $V^i(\omega_X,a)$;
% 	\end{enumerate}
% 	But none of the converses is necessarily true a priori.
% \end{rema}

We want to use the diagram \eqref{eqn:xy1} to prove Theorem \ref{thm:sm-in-codim1-sav} by reducing to the divisor case.
In order to do that, we need to prove that the fiber $Y_b$ satisfy the conditions of Corollary \ref{cor:divisorcase}.
Before, we need to study the base $Y_B$ and show, in particular, that $\chi(X_B,\omega_{X_B})=1$.

\begin{prop}\label{prop:base=1}
	Let $X\stackrel{f}\to Y$ be a desingularization of a subvariety $Y\subseteq A$ of an abelian variety and denote by $a: X\xrightarrow{f} Y\hookrightarrow A$ the induced morphism to $A$.
	Assume that $\chi(X,\omega_X)=1$ and $P+\PB$ is an irreducible component of $V^1(\omega_X,a)$ and let $i+1$ be the codimension of $\PB$ in $\PA$.

	Then, using notation as in Setting \ref{set:c1}, we have
	\begin{align*}
	\chi(X_B, R^i\beta_{*}(\omega_{X}\otimes a^*P))=1\qquad \text{ and }\qquad	\chi(X_B, \beta_*(\omega_X\otimes a^*P))=i+1.	
	\end{align*}
	 In particular, if $P+\PB$ passes through $0$, we may let $P=0$, and then $\chi(X_B, \omega_{X_B})=1$ and $\chi(X_B, \beta_*\omega_X)=i+1$.
	\end{prop}

\begin{proof}
	We use the notation in Setting \ref{set:c1}, in particular, we will mainly be using the following small part of diagram \eqref{eqn:xy1}
	\begin{equation*}
		\xymatrix{
		X\ar[r]\ar[d]_{\beta} & A\ar[d]\\
		X_B\ar[r]_{a_{B}}\ar[r]& B.}
	\end{equation*}
	
	We first apply the argument that Hacon--Pardini use to prove \cite[Lem.~3.2]{HP-g2}.

	We consider $Q$ to be a general point of $\PB$ and let $W\subset H^1(A, \cO_A)$ be a linear subspace of dimension $i+1$ transversal to $\PB$ at $Q$.
	Since $P+\PB$ is an irreducible component of $V^1(\omega_X,a)$, by the same argument as in \cite[Proof of Thm.~3]{el}, the derivative complex is exact besides at the leftmost term.
	Hence, the derivative complex induces the following exact sequence of vector bundles on $\mathbb{P}^i=\mathbb{P}(W)$:
	\begin{align}\label{eqn:reldercomp}
	0&\to \cO_{\mathbb{P}(W)}(-j)\to H^0(X, \omega_X\otimes a^*P\otimes a^*Q)\otimes \cO_{\mathbb{P}(W)}(-i)\to\cdots \to\\
	&\to H^{i-1}(X, \omega_X\otimes a^*P\otimes a^*Q)
	\otimes \cO_{\mathbb{P}(W)}(-1)\to H^i(X, \omega_X\otimes a^*P\otimes a^*Q)
	\otimes \cO_{\mathbb{P}(W)}\to 0,\notag
	\end{align}
	for some $j\geq i+1$.

	We note that since $Q\in\PB$ is general, we have the natural isomorphism $H^k(X, \omega_X\otimes a^*P\otimes a^*Q)\cong H^0(X_B, R^k \beta_*(\omega_X\otimes a^*P)\otimes a_B^*Q)$ for any $k\geq 0$.
	Hence, the exact sequence \eqref{eqn:reldercomp} becomes
	\begin{align}
	\label{eqn:sequence1}0&\to \cO_{\mathbb{P}(W)}(-j)\to H^0(X_B, \beta_*(\omega_X\otimes a^*P)\otimes a_B^*Q)\otimes \cO_{\mathbb{P}(W)}(-i)
	 \to\cdots \to\\
	&\to H^0(X_B, R^{i-1}\beta_*(\omega_{X}\otimes a^*P)\otimes a_B^*Q)
	 \otimes \cO_{\mathbb{P}(W)}(-1)\notag\\
	&\to H^0(X_B, R^i\beta_*(\omega_{X}\otimes a^*P)\otimes a_B^*Q)
	 \otimes \cO_{\mathbb{P}(W)}\to 0.\notag
	\end{align}
	Chasing through the above diagram, we get
	\begin{equation*}
		H^i(\mathbb{P}(W), \cO_{\mathbb{P}(W)}(-j))\cong H^0(X_B, R^i\beta_*(\omega_{X}\otimes a^*P)\otimes a_B^*Q)\otimes H^0(\mathbb{P}(W), \cO_{\mathbb{P}(W)}).
	\end{equation*}
	Hence
	\begin{equation}\label{eqn:egality1}
	\chi(X_B, R^i\beta_*(\omega_{X}\otimes a^*P))=h^0(X_B, R^i\beta_*(\omega_{X}\otimes a^*P)\otimes a_B^*Q)=\binom{j-1}{i}.
	 \end{equation}
	If we now tensor \eqref{eqn:sequence1} by $\cO_{\mathbb{P}(W)}(-1)$, chasing through the diagram we get
	\begin{equation*}
		H^i(\mathbb{P}(W), \cO_{\mathbb{P}(W)}(-j-1))\cong H^0(X_B, \beta_*(\omega_X\otimes a^*P)\otimes a_B^*Q)\otimes H^i(\mathbb{P}(W), \cO_{\mathbb{P}(W)}(-i-1)).
	\end{equation*}
	Hence
	\begin{equation}\label{eqn:egality2}
	 h^0(X, \omega_X\otimes a^*P\otimes a^*Q)=h^0(X_B, \beta_*(\omega_X\otimes a^*P)\otimes a_B^*Q)=\binom{j}{i}.
	 \end{equation}
	Combining \eqref{eqn:egality1} and \eqref{eqn:egality2}, we get
	\begin{equation*}\frac{h^0(X, \beta_*(\omega_X\otimes a^*P)\otimes a_B^*Q)}{h^0(X_B, R^i\beta_*(\omega_{X}\otimes a^*P)\otimes a_B^*Q)}=\frac{j}{j-i}\leq i+1,\end{equation*}
	and the equality holds only if $j=i+1$.

	\smallskip

	Let $d>1$ be the order of $P$ and $G$ the group generated by $P$.
	We then consider the \'{e}tale cover $A'\to A$ induced by $P$ and the corresponding base change diagram
	\begin{equation*}
	\xymatrix@C=12pt@R=12pt{
	&X\ar[rr]^(.5){f}&& A \\
	X'\ar[ur]^{\eta_X}\ar[rr]_{f'}&& A'.\ar[ur]^{\eta_A}
	}
	\end{equation*}
	Combining this diagram with \eqref{eqn:xy1} we obtain,
	\begin{equation*}
	\xymatrix@C=12pt@R=12pt{
	&X\ar[rr]\ar'[d][ddd]^{\beta}&& A\ar[ddd] \\
	X'\ar[ddd]_(.7){\beta'}\ar[ur]^{\eta_X}\ar[rr]&& A'\ar[ur]^{\eta_A}\ar[ddr]\\
	&\\
	&X_B\ar[rr]&& B,\\
	X'_B\ar[ur]^{\nu}\ar[urrr]_{\mu}
	}
	\end{equation*}
	where $X'\stackrel{\beta'}\to X'_B \stackrel{\nu}\to X_B$ is a birational modification of the Stein factorization of $X'\to X_B$ such that $X'_B$ is smooth.
	As in \cite[Proof of Lem.~3.3]{JLT}, we can arrange the modifications such that $\nu:X'_B\to X_B$ is still a generically finite $G$-cover and $\beta': X\to X'_B$ is $G$-equivariant.

	Let $V$ be the $(i+1)$-dimensional quotient vector space $H^0(A, \Omega_{A})/H^0(B, \Omega_B)$ and let $\cS$ be the saturation of ${\beta'}^*\Omega_{X'_B}\to \Omega_{{X'}}$ and $\cQ$ be the quotient of $\Omega_{{X'}}/\cS$.
	Then, we have on ${X'}$ the following commutative diagram of coherent sheaves:

	\begin{eqnarray*}
	\xymatrix@C=15pt@R=12pt{
	0\ar[r]& {\beta'}^*H^0(B, \Omega_B)\otimes \cO_{{X'}}\ar[r]\ar[d]& H^0(A, \Omega_{A})\otimes\cO_{{X'}}\ar[r]\ar[dd] &V\otimes \cO_{{X'}}\ar[r]\ar[dd]& 0\\
	& {\beta'}^*\Omega_{X'_B}\ar[d]\\
	0\ar[r]& \cS\ar[r] & \Omega_{{X'}}\ar[r] & \cQ \ar[r]& 0.
	}
	\end{eqnarray*}
	Considering $i$-th wedge of the right column, we have a morphism
	\begin{equation*} \bigwedge^iV\otimes \cO_{{X'}}\to \det\cQ\hookrightarrow\omega_{{X'}/X'_B},
	\end{equation*}
	and hence a morphism
	\begin{equation}\label{eqn:injective1}\bigwedge^iV\otimes \cO_{X'_B}\to \beta'_*\omega_{{X'}/X'_B}.
	\end{equation}

	Recall that, $Y_b$ is an irreducible ample divisor of $K_b$, by Lemma \ref{lem:propsett1}{\it (ii)}.
	By \cite[Corollary 3.11]{mor}, then the natural morphism \eqref{eqn:injective1} is injective.
	This gives us the following injective morphism between coherent sheaves on $X'_B$,
	\begin{equation*}\omega_{X'_B}^{\oplus(i+1)}\hookrightarrow {\beta'}_*\omega_{{X'}}.
	\end{equation*}
	Since $\beta'$ is $G$-equivariant, the above morphism is also $G$-equivariant.
	Hence, if we push-forward to $X_B$ and consider the direct summand corresponding to $P$, we have the following injective morphism of coherent sheaves on $X_B$,
	\begin{equation}\label{eqn:injective}R^i\beta_*(\omega_X\otimes a^*P)^{\oplus(i+1)}\hookrightarrow \beta_*(\omega_X\otimes a^*P).
	\end{equation}
	Hence, for any $Q\in\PB$, we have
	\begin{equation*}
	(i+1)h^0(X_B, R^i\beta_*(\omega_{X}\otimes a^*P)\otimes Q)\leq h^0(X_B, \beta_*(\omega_X\otimes a^*P)\otimes Q).
	\end{equation*}

	Therefore, we have $j=i+1$.
	So, from \eqref{eqn:egality1}, we get $\chi(X_B, R^i\beta_*(\omega_{X}\otimes a^*P))=\binom{i+1-1}{i}=1$ and, from \eqref{eqn:egality2}, $\chi(X_B, \beta_*(\omega_{X}\otimes a^*P))=\binom{i+1}{i}=i+1$.

	When $P=0$, we just note that $a^*P=\cO_X$ and $R^i\beta_*\omega_X=\omega_{X_B}$ by \cite[Prop.~7.6]{ko1}.
\end{proof}

The following proposition studies the general fibre $Y_b$.
More concretely, we want to prove that $\chi(X_b, \omega_{X_b})=1$ and $q(X_b)=i+1$.

\begin{prop}\label{prop:fiber=1}
Under the same assumptions as in Proposition \ref{prop:base=1} and using the notation of Setting \ref{set:c1},
we consider in addition that $Y_B$ is of general type.

Then a general fiber $X_b$ of $\beta$ has $\chi(X_b, \omega_{X_b})=1$ and $q(X_b)=i+1$.
\end{prop}

\begin{proof}
	We use again the notation in Setting \ref{set:c1}, in particular we recall diagram \eqref{eqn:xy1}
	\begin{equation*}
		\xymatrix{
		X_b\ar@{^{(}->}[d]\ar[r]& Y_b \ar@{^{(}->}[d] \\
		X\ar[r]\ar[d]_{\beta} & Y\ar[d]\ar@{^{(}->}[r] & A\ar[d]\\
		X_B\ar[r]\ar@/_1.5pc/[rr]_{a_{B}}& Y_B\ar@{^{(}->}[r] & B.}
	\end{equation*}

	First, we claim that the injective morphism \eqref{eqn:injective} is generically an isomorphism.
	Otherwise, we have a short exact sequence of coherent sheaves:
	\begin{equation*}0\to R^i\beta_*(\omega_X\otimes a^*P)^{\oplus(i+1)}\to \beta_*(\omega_X\otimes a^*P)\to \cW\to 0,\end{equation*}
	where $\rk\cW\geq 1$.
	Then, we have a short exact sequence of non-trivial coherent sheaves on $B$,
	\begin{equation*}0\to a_{B*}R^i\beta_*(\omega_X\otimes a^*P)^{\oplus(i+1)}\to a_{B*} \beta_*(\omega_X\otimes a^*P)\to a_{B*}\cW\to 0.\end{equation*}
	Since $Y_B$ is of general type, $ a_{B*}\beta_*(\omega_X\otimes a^*P)$ is M-regular on $B$ (see Lemma \ref{lem:hypo3}).
	Therefore, $a_{B*}\cW$ is a non-trivial M-regular sheaf on $B$ and $V^0(a_{B*}\cW)=\PB$ (see Remark~\ref{rem:chi>0-Xgt}{\it(ii)}).
	However, for $Q\in \PB$ general, $h^0(X_B, \cW\otimes a_B^*Q)=h^0(X_B, \beta_*(\omega_X\otimes a^*P)\otimes a_B^*Q)-(i+1)h^0(X_B, R^i\beta_*(\omega_X\otimes a^*P)\otimes a_B^*Q)=0$, which contradicts $V^0(a_{B*}\cW)=\PB$.

	Thus, for a general fiber $X_b$ of $\beta$, we have $h^0(X_b, \omega_{X_b})=\rk \beta_*(\omega_X\otimes a^*P)=i+1$ and then, by \cite[Corollary 3.11]{mor}, $\chi(X_b, \omega_{X_b})=1$ and $q(X_b)=i+1$.
\end{proof}

Now we are ready to prove Theorem \ref{thm:sm-in-codim1-sav} by induction on the number of irreducible components of $\Vz^1(\omega_X,a_X)$.

\begin{proof}[Proof of Theorem \ref{thm:sm-in-codim1-sav}]
	Note that $Y$ is of general type by Remark \ref{rem:chi>0-Xgt}{\it(i)}.
	When $\codim_A Y=1$, we are in the divisor case and we conclude by Corollary \ref{cor:divisorcase}.
	In particular $V^1(\omega_X,a)=\set{0}$.

	So we may assume $\codim_A Y\geq 2$. We consider the Albanese morphism of $a_X:X\to A_X$, and we denote by $M=a_X(X)$ and $t: A_X\to A$ the induced morphism between abelian varieties.
	Then we have the commutative diagram
	\begin{equation}\label{eqn:lleig}
		\xymatrix@!C=30pt@R=12pt{
		X\ar[dr]^g\ar[rr]^{f}&& Y\ar@{^{(}->}[rr]&&A \\
		& M\ar[ur]_(.6){t_Y}\ar@{^{(}->}[rr]&& A_X\ar[ur]_(.6)t,}
	\end{equation}
	where $a:X\to Y\hookrightarrow A$ and $a_X:X\to M\hookrightarrow A_X$. Since $f$ is birational, $g$ and $t_Y$ are also birational.
	As we already mentioned, we assume that $Y$ generates $A$, so $\widehat{t}: \PA\to \PA_X$ is an isogeny onto its image

	By Lemma \ref{lem:V^1}, any irreducible component of codimension-$(i+1)$ of $V^1(\omega_X,a)$ (respectively of $V^1(\omega_X,a_X)$) is an irreducible component of $V^i(\omega_X,a)$ (resp. $V^i(\omega_X,a_X)$).
	Thus, all irreducible components of $V^1(\omega_X,a)$ and $V^1(\omega_X,a_X)$ are positive dimensional, since $\codim_{A_X} M\geq \codim_A Y\geq 2$.
	Hence, the conditions in Setting \ref{set:c1} are fulfilled and we use the notation therein.

	By Lemma \ref{lem:pass0} (see also Remark \ref{rem:V1pass0}), each codimension-$(i+1)$ component $\PB$ of $V^i(\omega_X, a)$ passes through $0$ and
	$X_B$ and $Y_B$ are of general type.
	
    We denote by $\Vz^1(\omega_X, a_X)$ the union of irreducible components of $V^1(\omega_X, a_X)$ passing through $0$.
	We denote by $k$ the number of positive dimensional irreducible components of $\Vz^1(\omega_X, a_X)$,
	\begin{equation*}
		\Vz^1(\omega_X, a_X)=\bigcup_{1\leq j\leq k}\PPB_j,
	\end{equation*}
    and $\codim_{\PA_X}\PPB_j=i_j+1$.

	We consider $\PB_j=\widehat{t}^{-1}(\PPB_j)$.
	Then, $\codim_{\PA}\PB_j\leq i_j+1$.
	Moreover, since $M$ is of general type, then $\PPB_j$ is an irreducible component of $V^{i_j}(\omega_X, a_X)$ by Lemma \ref{lem:V^1}.
	Thus, $\PB_j\subset V^{i_j}(\omega_X, a)$.
	Since $Y$ is also of general type, then $\codim_{\PA}\PB_j\geq i_j+1$ by Lemma \ref{lem:hypo3}.
	Thus, $\PB_j$ is a component of $V^{i_j}(\omega_X, a)$ of codimension $i_j+1$ in $\PA$.

	We first consider Setting \ref{set:c1} with respect to $\widetilde{B}_1$ and $B_1$, irreducible components of $V^{i_1}(\omega_X, a_X)$ and $V^{i_1}(\omega_X, a)$ of codimension $i_1+1$ in $\PA_X$ and $\PA$.
	Using the perspective of \eqref{eqn:lleig} and \eqref{eqn:xy1}, we have
	\begin{equation*}
		\xymatrix@!C=30pt@R=12pt{
		X_{b_1}\ar@{^{(}->}[dd]\ar[rr]^{f_{b_1}}&& Y_{b_1} \ar@{^{(}->}[dd] \\ &\\
		X\ar[dd]_{\beta_{1}}\ar[dr]_g\ar[rr]^{f}&& Y \ar@{^{(}->}[rr]\ar'[d][dd]&& A \ar[dd] \\
		& M\ar[ur]_(.6){t_Y}\ar[dd]\ar@{^{(}->}[rr]&& A_X \ar[ur]_(.6)t\ar[dd]\\
		X_{\widetilde{B}_1}\ar[dr]_{g_1}\ar'[r]^{f_{B_1}}[rr]&& Y_{B_1}\ar@{^{(}->}'[r][rr]&& B_1\\
		& M_{\widetilde{B}_1}\ar[ur]\ar@{^{(}->}[rr] && \widetilde{B}_1\ar[ur],}
	\end{equation*}
	where $X_{\widetilde{B}_1}$ is smooth projective and $\beta_1: X\rightarrow X_{\widetilde{B}_1}$ is a birational model of the Stein factorization of
    $X\rightarrow M\rightarrow M_{\widetilde{B}_1}$.
	Then $\beta_1$ is also a model of the Stein factorization of $X\rightarrow Y\rightarrow Y_{B_1}$, so in the notation of Setting \ref{set:c1}, $X_{\widetilde{B}_1}=X_{B_1}$.
	Note that, since $Y$ is smooth in codimension $1$, by Lemma \ref{lem:pass0} we have that $Y_{B_1}$ is of general type and $f_{B_1}: X_{\widetilde{B}_1}\to Y_{B_1}$ is a birational morphism.
	Thus, $g_1: X_{\widetilde{B}_1}\to M_{\widetilde{B}_1}$ is also a birational morphism.

	Since $M_{\widetilde{B}_1}$ is of general type, we can apply Proposition \ref{prop:fiber=1} to $X\rightarrow M\rightarrow M_{\widetilde{B}_1}$.
	Hence, for a general point $b_1\in X_{\widetilde{B}_1}$, and $X_{b_1}$ the fiber of $\beta_1$ over $b_1$, we have $\chi(X_{b_1}, \omega_{X_{b_1}})=1$.

	Since $f$ is birational, $f_{b_1}=\res{f}{X_{b_1}}$ is a birational morphism  onto the corresponding fiber $Y_{b_1}$ of $Y\rightarrow Y_{B_1}$. Moreover, since $Y$ is smooth in codimension $1$, so is $Y_{b_1}$.
	Let $K_1$ be the kernel of $A\to B_1$.
	Then, by Lemma \ref{lem:propsett1}{\it (ii)}, $Y_{b_1}$ is an irreducible smooth-in-codimension-$1$ ample divisor of a translate of $K_1$.
	We conclude by Corollary \ref{cor:divisorcase} that $Y_{b_1}$ is a theta divisor $\Theta_1$ of $K_1\subset A$.
	In particular, $X_{b_1}$ is birational to $\Theta_1$.
	By Proposition \ref{prop:fibertheta}, we can assume that $X=X_{b_1}\times X_{\widetilde{B}_1}$. Thus, we have
	\begin{equation}\label{eqn:1stprod}
		A_X=K_1\times A_{X_{\widetilde{B}_1}}\qquad \text{ and } \qquad M=\Theta_1\times M_1,
	\end{equation}
	where we define $M_1:=M_{X_{\widetilde{B}_1}}=a_{X_{\widetilde{B}_1}}(X_{\widetilde{B}_1})$.
	Hence, we also have that $\PPB_1=\set{0}\times \PA_{X_{\widetilde{B}_1}}$ and
	\begin{equation*}
		\Vz^1(\omega_X, a_X)=\Big(\set{0}\times \PA_{X_{\widetilde{B}_1}}\Big)\cup \Big(\PK_1\times \Vz^1(\omega_{X_{\widetilde{B}_1}}, a_{X_{\widetilde{B}_1}})\Big).
	\end{equation*}
	Therefore $\PPB_2=\PK_1\times \PB'_{2}$ for some abelian subvariety $\PB'_{2}\subseteq \Vz^1(\omega_{X_{\widetilde{B}_1}}, a_{X_{\widetilde{B}_1}})\subset \PA_{X_{\widetilde{B}_1}}$.

    We now consider Setting \ref{set:c1} with respect to $\widetilde{B}_2$ and $B_2$, irreducible components of codimension $i_2+1$.
	Using the perspective of \eqref{eqn:lleig}, \eqref{eqn:xy1}, as before, and adding the decomposition \eqref{eqn:1stprod}, we have
	\begin{equation*}
	\xymatrix@!C=30pt@R=12pt{
	X_{b_2}\ar@{^{(}->}[dd]\ar[rr]^{f_{b_2}}&& Y_{b_2} \ar@{^{(}->}[dd] \\ &\\
	X=X_{b_1}\times X_{\widetilde{B}_1}\ar[dd]_{\beta_{2}}\ar[dr]_g\ar[rr]^{f}&& Y \ar@{^{(}->}[rr]\ar'[d][dd]&& A \ar[dd] \\
	& \Theta_1\times M_1\ar[ur]_(.6){t_Y}\ar[dd]\ar@{^{(}->}[rr]&& K_1\times A_{X_{\widetilde{B}_1}} \ar[ur]_(.6)t\ar[dd]\\
	X_{\widetilde{B}_2}\ar[dr]_{g_2}\ar'[r]^{f_{B_2}}[rr]&& Y_{B_2}\ar@{^{(}->}'[r][rr]&& B_2\\
	& \Theta_1\times M'_{2}\ar[ur]\ar@{^{(}->}[rr] && \widetilde{B}_2=K_1\times B'_{2}\ar[ur],}
	\end{equation*}
	where we have denoted by $M'_{2}$ the image of the composition of morphism $M_1\to A_{X_{\widetilde{B}_1}}\to B'_{2}$.
	Note that $f_{B_2}$ and $g_2$ are birational and $Y_{B_2}$ is of general type by Lemma \ref{lem:pass0}.
 	
	Let $K_2$ be the kernel of $A\to B_2$.
	As before, Lemma \ref{lem:propsett1}, Proposition \ref{prop:fiber=1}, and Corollary \ref{cor:divisorcase} imply that a general fiber $X_{b_2}$ of $\beta_{2}$ is birational to a theta divisor $\Theta_2$ of $K_2\subset A$.
	Then, by Proposition \ref{prop:fibertheta}, $X_{b_1}\times X_{\widetilde{B}_1}$ is birational to $X_{b_2}\times X_{\widetilde{B}_2}$.
	Moreover, the morphism $g_2:X_{\widetilde{B}_2}\rightarrow \Theta_1\times M'_2$ induces a birational equivalence between $X_{\widetilde{B}_2}$ and the product of $X_{b_1}$ with another variety that we call $X_{2}$.
	That is, $X$ is birational to $X_{b_1}\times X_{b_2}\times X_{2}$ and we have
	\begin{equation*}\label{eqn:2stprod}
		A_X=K_1\times K_2 \times A_{X_{2}}\qquad \text{ and } \qquad M=\Theta_1\times \Theta_2 \times M_2,
	\end{equation*}
	where $M_2=a_{X_{2}}(X_{2})$.
	Hence, we also have that $\PPB_2=\PK_1\times \set{0}\times \PA_{X_{2}}$ and
	\begin{equation*}
		\Vz^1(\omega_X, a_X)=\Big(\set{0}\times \PA_{X_{\widetilde{B}_1}}\Big)\cup \Big(\PK_1\times \set{0}\times \PA_{X_{2}}\Big)\cup \Big(\PK_1\times \PK_2 \times \Vz^1(\omega_{X_{2}}, a_{X_{2}})\Big).
	\end{equation*}
	Therefore, $\PPB_3=\PK_1\times \PK_2 \times \PB'_{3}$ for some abelian subvariety $\PB'_{3}\subseteq \Vz^1(\omega_{X_{2}}, a_{X_{2}})\subset \PA_{X_{2}}$.

\medskip

	We can iterate this process using all $\PPB_j$ for $1\leq j\leq k$.
	For each fibration $\beta_{j}: X\to X_{\widetilde{B}_j}$, we get a further decomposition of $X$.
	Finally, we conclude that $X$ is birational to $X_{b_1}\times \ldots \times X_{b_k}\times X_{k}$,
	\begin{equation*}\label{eqn:genprod}
		A_X=K_1\times \ldots \times K_k \times A_{X_{k}}\qquad \text{ and } \qquad M=\Theta_1\times\ldots \times \Theta_k \times M_k,
	\end{equation*}
	where $M_k=a_{X_{k}}(X_{k})$.
	Hence, we have that
	\begin{align*}
		\Vz^1(\omega_X, a_X)=&\Big(\set{0}\times \PA_{X_{\widetilde{B}_1}}\Big)\cup \Big(\PK_1\times \set{0}\times \PA_{X_{2}}\Big)\cup \Big(\PK_1\times \PK_2 \times \set{0}\times \PA_{X_{3}}\Big)\cup\\
		&\cup\ldots\cup \Big(\PK_1\times \cdots \times \PK_{k-1} \times \set{0}\times \PA_{X_{k}}\Big)\cup\\
		&\cup \Big(\PK_1\times \cdots \times \PK_{k}\times \Vz^1(\omega_{X_{k}}, a_{X_{k}})\Big).
	\end{align*}

	To conclude, we only need to show that $M_k$ is a point.
	Suppose by contradiction that $\dim M_k\geq 1$.
	Then $\dim X_{k}=\dim M_k\geq 1$ and hence $V^1(\omega_{X_{k}}, a_{X_{k}})$ contains $0$.
	Then, $\PK_1\times\cdots\times \PK_k\times V^1(\omega_{X_k}, a_{X_k})$ exhibits a new component of $\Vz^1(\omega_X, a_X)$, which is a contradiction.
	
	Hence, $X$ is birational to a product of $k$ theta divisors and $A_X=K_1\times \cdots \times K_k$.
\end{proof}

Once we have proven Theorem \ref{thm:sm-in-codim1-sav}, we can deduce Theorem \ref{thmA} from it.
\begin{proof}[Proof of Theorem \ref{thmA}]
	Only one direction needs to be proved.

	We fix a desingularization $f: X\to Y$.
	Then by assumption, $\chi(X, \omega_X)=1$.
	As we already said, we assume that $Y$ generates $A$.
	By Theorem \ref{thm:sm-in-codim1-sav}, $X$ is birational to a product of theta divisors and the induced morphism $t: A_X\to A$ is surjective.

	The image $a_X(X)$ is a product of theta divisors $\Theta_1\times\cdots\times \Theta_m$.
	We denote by $K_j$ the Albanese variety $A_{\Theta_j}$ of $\Theta_j$.
	In order to show that the induced morphism $\tau: \Theta_1\times\cdots\times \Theta_m\to Y$ is an isomorphism, we just need to prove that $t:A_X=K_{1}\times\cdots \times K_{m}\to A$ is an isomorphism.
        We will show that if $t$ is not an isomorphism, $Y$ can not be normal.

\medskip

 	First of all, we assume that $t$ is an isogeny of degree $d>1$.
 	Then $t^{-1}(Y)$ has $d$ irreducible components, each of which is a translate of the product of theta divisors $\Theta_1\times\cdots\times \Theta_m$ in $A_X$.
	Theta divisors are ample divisors, so the $d$ irreducible components meet.
	Thus, $t^{-1}(Y)$ is connected but not irreducible, hence not normal.
	This gives a contradiction, since $Y$ is normal and $t$ is \'etale.

\medskip
	
	Now we assume that $\dim \ker t>0$.
	The strategy is to prove that there always exist non-connected fibers of $\tau$ to get a contradiction with Zariski's Main Theorem.

	Let $K$ be  the neutral component of $\ker t$.
	We denote by $k=\dim K$ and $\varphi_j: K\to K_j$ the natural projection.
	By induction on the number of theta divisors, we can assume that $\varphi_j\neq 0$ for all $j$.
	Let $t_1: A_X\to A_X/K$ be the quotient and $\overline{Y}$ the image $t_1(a_X(X))$.
	Then we have the following commutative diagram
	\begin{equation*}
		\xymatrix{
		\Theta_1\times\cdots\times\Theta_m \ar@{^{(}->}[r]\ar@/_2pc/[dd]_{\tau}\ar[d]^{\tau_1} & K_1\times\cdots\times K_m\ar[d]^{t_1}\ar@/^2pc/[dd]^t\\
		\overline{Y} \ar@{^{(}->}[r]\ar[d]^{\tau_2} & A_X/K\ar[d]^{t_2}\\
		Y \ar@{^{(}->}[r] & A,}
	\end{equation*}
	Note that, since $Y\subsetneq A$, we have $k=\dim K< m$.

	For $c=(c_1, \ldots, c_m)\in \Theta_1\times\cdots\times\Theta_m$, we have $\tau_1^{-1}\tau_1(c)=(K+c)\cap (\Theta_1\times\cdots\times\Theta_m)$.
	Thus,
	\begin{equation*}
		\tau_1^{-1}\tau_1(c)\cong\varphi_1^{-1}(\Theta_1-c_1)\cap\cdots \cap \varphi_m^{-1}(\Theta_m-c_m).
	\end{equation*}
	We observe the following:
	\begin{itemize}
		\item[1)] $\bigcup_{c_j\in\Theta_j}(\Theta_j-c_j)=(\Theta_j-\Theta_j)=K_j$, since $\dim K_j>1$;
		\item[2)] Since $\tau$ is birational, for $c\in \Theta_1\times\cdots\times\Theta_m$ general, we have $\bigcap_{j=1}^m\varphi_j^{-1}(\Theta_j-c_j)=\set{0_K}$ the identity element of $K$.
	\end{itemize}
	Thanks to the second observation, there exists $1\leq t\leq m-1$ and $d_{j}\in\Theta_{j}$ for $1\leq j\leq t$ such that,
	\begin{equation*}
		\begin{cases}
			\dim\big(\bigcap_{i=1}^t\varphi_{i}^{-1}(\Theta_{i}-d_{i})\big)=1 &\text{and}\\
			\dim\big(\bigcap_{i=1}^{t+1}\varphi_{i}^{-1}(\Theta_{i}-d_{i})\big)=0 & \text{for a general }d_{t+1}\in \Theta_{t+1}.
		\end{cases}
	\end{equation*}

\medskip

    We first claim that $\bigcap_{i=1}^t\varphi_{i}^{-1}(\Theta_{i}-d_{i})$ is connected.
	Otherwise, let $\nu\in\bigcap_{i=1}^t\varphi_{i}^{-1}(\Theta_{i}-d_{i})$ be a point not in the connected component containing $0_K$.
	We then choose $d_j'\in \Theta_j\cap(\Theta_j-\varphi_j(\nu))$ for $t+1\leq j\leq m$.
	Then
	\begin{equation*}
		\set{0_K, \nu}\subset \Big(\bigcap_{i=1}^t\varphi_{i}^{-1}(\Theta_{i}-d_{i})\Big)\cap\Big(\bigcap_{j=t+1}^m\varphi_j^{-1}(\Theta_j-d_j')\Big)    \subset\bigcap_{i=1}^t\varphi_{i}^{-1}(\Theta_{i}-d_{i}).
	\end{equation*}
    Thus, $\big(\bigcap_{i=1}^t\varphi_{i}^{-1}(\Theta_{i}-d_{i})\big)\cap\big(\bigcap_{j=t+1}^m\varphi_j^{-1}(\Theta_j-d_j')\big)$ is also not connected. We denote $$d'=(d_1,\ldots,d_t, d_{t+1}',\ldots,d_{m}').$$ Then $\tau_1^{-1}\tau(d')$ is not connected. Since $t_2$ is an isogeny, $\tau_2$ is finite. Hence, $\tau_2^{-1}(\tau(d'))$ is a set of finite points containing $\tau_1(d')$. Thus, $\tau^{-1}\tau(d')$ is also not connected.
	This contradicts Zariski's Main Theorem, since $\tau$ is birational and $Y$ is normal.

\medskip

	We now consider $\bigcap_{i=1}^t\varphi_{i}^{-1}(\Theta_{i}-d_{i})$ with the reduced scheme structure and we denote it by $C$.
	Hence, $C$ is a connected curve containing $0_K$. Then $\varphi_{t+1}(C)$ is a curve in $K_{t+1}$ containing the identity element of $K_{t+1}$.
	We have the following Zariski open dense subset of $\Theta_{t+1}$:
	\begin{equation*}
		U:=\set{u\in\Theta_{t+1}\mid \dim \big(C\cap \varphi_{t+1}^{-1}(\Theta_{t+1}-u)\big)=0}.
	\end{equation*}
	Considering the following set
	\begin{equation*}
		V:=\bigcup_{\begin{smallmatrix} c\in \varphi_{t+1}(C)\\ c\neq 0_{t+1}\end{smallmatrix}}(\Theta_{t+1}-c)=\Theta_{t+1}-\big(\varphi_{t+1}(C)\setminus \set{0_{t+1}}\big),
	\end{equation*}
    we	have
    \begin{equation*}
		V\cap \Theta_{t+1}=\bigcup_{\begin{smallmatrix} c\in \varphi_{t+1}(C)\\ c\neq 0_{t+1}\end{smallmatrix}}\Big(\Theta_{t+1}\cap(\Theta_{t+1}-c)\Big).
	\end{equation*}
	Since $\Theta_{t+1}$ is a theta divisor of $K_{t+1}$,  the induced morphism
	\begin{eqnarray*}&K_{t+1}&\rightarrow \Pic^0K_{t+1}\\
		& c&\rightarrow t_c^*\cO_{K_{t+1}}(\Theta_{t+1})\otimes\cO_{K_{t+1}}(-\Theta_{t+1})
	\end{eqnarray*}
	is an isomorphism.
	Moreover, since $\Theta_{t+1}$ has rational singularities, $\Pic^0K_{t+1}\rightarrow\Pic^0\Theta_{t+1}$ is an isomorphism.
	Thus, $\res{(\Theta_{t+1}-c)}{\Theta_{t+1}}$ are pairwise different effective divisors of $\Theta_{t+1}$ for different $c\in \varphi_{t+1}(C)\setminus \set{ 0_{t+1}}$.
	Therefore $V\cap \Theta_{t+1}$ is a Zariski dense subset of $\Theta_{t+1}$.
	
	Hence $U\cap V$ is also a Zariski dense subset of $\Theta_{t+1}$.
	Given any $d_{t+1}\in U\cap V\subset \Theta_{t+1}$, the set
	\begin{equation*}
		\big(\Theta_{t+1}-d_{t+1}\big)\cap \big(\varphi_{t+1}(C)\big)
	\end{equation*}
	is finite with at least $2$ distinct points.
	Thus,
	\begin{equation*}
		\bigcap_{j=1}^{t+1}\varphi_{j}^{-1}(\Theta_{j}-d_{j}) \supset \Big( \varphi_{t+1}^{-1}(\Theta_{t+1}-d_{t+1})\cap C\Big)
	\end{equation*}
	is also a finite set containing at least $2$ distinct points $0_K$ and $\kappa$.

	By observation $1)$, we can choose $d_j\in \Theta_j$ for $t+1\leq j\leq m$ such that $\kappa\in \varphi_j^{-1}(\Theta_j-d_j)$ for any  $t+1\leq j\leq m$.
	Let $d=(d_1,\ldots,d_m)\in \Theta_1\times\cdots\times\Theta_m$.
	Then
	\begin{equation*}
		\tau_1^{-1}\tau_1(d)\cong \varphi_1^{-1}(\Theta_1-d_1)\cap\cdots \cap \varphi_m^{-1}(\Theta_m-d_m)
	\end{equation*}
	is a finite set with at least two distinct points.
	In particular, $\tau_1^{-1}\tau_1(d)$ is not connected and hence $\tau^{-1}\tau(d)$ is also not connected, which is a contradicts to Zariski's Main Theorem.
\end{proof}

As we have seen in the previous proof the difference between Theorem \ref{thmA} and Theorem \ref{thm:sm-in-codim1-sav} is quite subtle.
The following examples may help to distinguish the two situations.
\begin{exam}
	Let $X$ be a desingularization of $\Theta_1\times\Theta_2\subset A_1\times A_2=A$.
	Then, the Albanese variety of $X$ is $A=A_1\times A_2$ and the image $a_X(X)=Y$ is $\Theta_1\times\Theta_2$, which is normal.

	Now, let $P_1\in A_1$ and $P_2\in A_2$, be two $2$-torsion points.
	Consider $t:A_1\times A_2\to A'$ the isogeny induced by $(P_1,P_2)$ (just the quotient by the finite group generated by $(P_1,P_2)$).
	Indeed, $t:A_1\times A_2\to A'$ is an \'etale covering of degree $2$.
	Let $Y'\subset A'$ be $t(\Theta_1\times \Theta_2)$.
	Note that $Y'$ is smooth in codimension $1$.
	But $Y'$ is not normal, since $t^{-1}(Y')=\Big(\Theta_1\times\Theta_2 \Big)\cup \Big((\Theta_1+P_1)\times (\Theta_2+P_2)\Big)$ is connected but not irreducible, hence not normal.

	Thus, $Y'\subset A'$ exhibits a case $Y'$ is birational to a product of theta divisors, but not isomorphic.
	Moreover, $A'$ is not necessarily isomorphic to the Albanese variety of its desingularization.
	
	Summing up, in the following diagram, the first row is a paradigmatic example of Theorem \ref{thmA}, while the second row is a paradigmatic example of Theorem \ref{thm:sm-in-codim1-sav} when $t$ is an isogeny:
	\begin{equation*}
		\xymatrix{
		X\ar[r]_(.3)f\ar@{=}[d]\ar@/^1.5pc/[rr]^{a_X} & Y=\Theta_1\times \Theta_2 \ar[d]^{\mathrm{bir}}\ar@{^{(}->}[r] & A=A_X\ar[d]_{t}^{2:1}\\
		X\ar@/_1.5pc/[rr]_{a'}\ar[r]^(.3){f'} & Y'=t(\Theta_1\times \Theta_2) \ar@{^{(}->}[r] & A'.
		}
	\end{equation*}
\end{exam}

\begin{exam}
	Let $C\to E$ a bielliptic curve, such that $g(C)=2$ and let $K$ be the kernel of $JC=A_C\to E$.
	Then consider $X=C\times C\times C$ and let
	\begin{equation*}
		\xymatrix{
		X\ar@{=}[r]_f\ar@{=}[d]\ar@/^1.5pc/[rr]^{a_X} & Y\ar[d]^{\mathrm{bir}}\ar@{^{(}->}[r] & A_X\ar[d]^{t}\\
		X\ar@/_1.5pc/[rr]_{a'}\ar[r]^(.4){f'} & Y'=t(Y) \ar@{^{(}->}[r] & A',
		}
	\end{equation*}
	where $t:A_X=JC\times JC\times JC \to A$ is the quotient by the diagonal of $K$, \ie\ $\set{ (x, x, x)\in A_X \mid x\in K}$.
	Note that $\res{t}{Y}$ is a birational morphism and $Y'$ is smooth in codimension $1$.
%, namely it is singular along the image via $t$ of $\set{ (c,c,c)\in X \mid c\in C}$ and if the
%  bielliptic involution and the hyperelliptic involution commute, there exist other singularities

	This example shows that, in Theorem \ref{thm:sm-in-codim1-sav}, $t$ does not need to be an isogeny.
\end{exam}

\medskip

Finally, we want to mention the following difference between Theorems \ref{thmA} and Theorem \ref{thm:sm-in-codim1-sav}.
\begin{rema}\label{rem:ppav_gencase}
	On the one hand, as we have already seen in the previous examples, the abelian variety in Theorem \ref{thm:sm-in-codim1-sav} does not need to be principally polarized (see also Remark \ref{rem:notppav_divcase}).

	On the other hand, we emphasize that in Theorem \ref{thmA} the abelian variety is always principally polarized.
	Indeed, $Y$ induces a principal polarization on $A$ and $A=A_X$.
\end{rema}

\bigskip

\section{Varieties of maximal Albanese dimension \& holomorphic Euler characteristic 1}
\label{sec:mAd}

In the previous section, we have considered $a:X\to A$ generically finite morphisms from a smooth projective variety $X$ to a subvariety of an abelian variety $Y:=a(X)\subset A$, from the point of view of $Y$, letting $X$ be its desingularization.
Now we will consider that $X$ is given and $a$ will be its Albanese morphism.
%So, $X$ will be a smooth projective variety and we will consider $a_X:X\to A_X$ its Albanese morphism and $Y=a_X(X)$.

\medskip

If we consider $a_X:X\stackrel{f}\to Y:=a_X(X)\hookrightarrow A_X$, we note, even if $X$ is of general type, $f:X\to Y$ might not be birational. We have two paradigmatic examples of smooth complex projective varieties $X$ of maximal Albanese dimension and general type, such that $f$ is not birational.
The first one is when $X$ is of general type and $\chi(X,\omega_X)=0$.
The second one is when $\chi(X,\omega_X)=1$ and $q(X)=\dim X$.
As we have mentioned in the introduction, the complete classification of these examples seems to be a rather difficult problem
and clearly the knot comes from studying $f:X\to Y=a_X(X)$ rather than just $a_X(X)$ as we have done in the previous section.
% Smooth projective varieties of maximal Albanese dimension, general type and $\chi(X,\omega_X)=0$ exist in dimension greater than $2$.
% The first, somehow unexpected, examples were constructed by Ein-Lazarsfeld and they are the desingularization of the quotient $(C_1\times C_2\times C_3)/(i_1\times i_2\times i_3)$ of the product of $3$ bielliptic curves $C_i$ of genus $2$ modulo the diagonal involution (see \cite[Ex.~1.13]{el}).
%
% Chen--Debarre--Jiang have studied in \cite{CDJ} these varieties, \ie\ smooth complex projective varieties $X$ of maximal Albanese dimension and of general type satisfying $\chi(X,\omega_X)=0$.
% They show that in this case the Albanese variety of $X$ needs to have at least three simple factors, justifying thus their existence only in dimension $3$ or higher.
% Moreover, they show that in dimension three, up to abelian \'etale covers, the Ein--Lazarsfeld examples are the only ones.
% In \cite{CDJ}, the authors speculate and give partial results to classify these varieties in any dimension, but assuming that the Albanese variety has only three simple factors.
% Without this last assumption, the problem seems to be rather open and even more difficult.

\medskip

Now, $X$ will be always a smooth complex projective variety such that
\begin{align}
&X \text{ is of maximal Albanese dimension;}\tag{C.1}\label{hypo1}\\
&\chi(X, \omega_X)=1.\label{hypo2}\tag{C.2}
\end{align}
Then Hacon--Pardini show in \cite{HP-g2} that
\begin{equation*}
	\dim X \leq q(X) \leq 2\dim X,
\end{equation*}
and, when $q(X)$ is maximal, \ie\ $q(X)=2\dim X$, then Hacon--Pardini prove that $X$ is birational to a product of curves of genus $2$ (theta divisors of dimension 1).
In particular, $f$ is birational.

\medskip

We want to extend the classification Theorem of Hacon--Pardini to lower values of $q(X)$.
First, we will give a complete classification of the submaximal case, \ie\ varieties satisfying \Co\ and $q(X)=2\dim X-1$.
Already in this case the examples where $f$ is not birational will step into our classification problem.
More precisely, we will have two cases depending whether $f$ is birational or it has degree $2$.

Second, we will study how to characterize birationally products of higher dimensional theta divisors.
For this, we will need to impose the following extra conditions
\begin{align}
&a_X(X) \text{ is of general type;}\label{hypo3}\tag{C.3}\\
\tag{C.4} &a_X(X) \text{ is smooth in codimension }1. \label{hypo4}
\end{align}
where \eqref{hypo3} prevents $f$ from being non birational and it is an open problem to determine
if condition \eqref{hypo4} is necessary (see Question \ref{qu?}).

\bigskip

\def\thesubsection{\thesection.1}

\subsection*{Submaximal irregularity}%\label{ssec:submax}
The case when $q(X)=2\dim X -1$, was already studied for surfaces, \ie\ $\dim X= 2$ and $q(X)=p_g(X)=3$, by Hacon--Pardini \cite{HP-surf} and Pirola \cite{pirola}.
This section will be devoted to prove the following theorem, which can be seen as a generalization of both \cite{HP-surf,pirola} and \cite{HP-g2}.
\begin{theo}\label{thm:2n-1main}
	Let $X$ be a smooth projective variety of maximal Albanese dimension.
	Assume that $\chi(X,\omega_X)=1$ and $q(X)=2\dim X-1$.
	Then $X$ is birational to one of the following
	varieties:
	\begin{enumerate}
		 \item a product of smooth curves of genus $2$ with the $2$-dimensional symmetric product of a curve of genus $3$ (\ie\ $X$ is a product of theta divisors);
		 \item $(C_1\times \widetilde{Z})/\langle\tau\rangle$, where
			$C_1$ is a bielliptic curve of genus $2$,
			$\widetilde{Z}\to C_1\times\cdots \times C_{n-1}$ is an \'etale double cover of a product of smooth projective curves of genus $2$, and
			$\tau$ is an involution acting diagonally on $C_1$ and $\widetilde{Z}$ via the involutions corresponding respectively to the double covers.
	 \end{enumerate}
\end{theo}

The next natural step would be to classify varieties satisfying \Co\ and $q(X)=2\dim X -2$, but we note that this classification would include the classification of surfaces $S$ with $q(S)=p_g(S)=2$, which is not complete.

\medskip

To prove the previous theorem, we need the following general proposition, which can be seen as a variant of the main Theorem in \cite[Thm.~3.1]{HP-g2}, but it is also the main technical tool to prove Theorem \ref{thm:2n-1main}.
\begin{prop}\label{prop:hpTech}
	Let $f: V\to Z$ be a surjective morphism between smooth projective varieties, where $Z $ is of maximal Albanese dimension.
	Let $\cF$ be a direct summand of $f_*\omega_V$, such that $\chi(Z, \cF)=1$.
	Then,
	\begin{equation*}
		q(Z)\leq 2\dim Z.
	\end{equation*}
	Moreover, if equality holds, then
	\begin{enumerate}
		\item $Z$ is birational to a product of smooth curves of genus $2$;
		\item the Albanese morphism factors through $a_Z: Z\xrightarrow{g} C_1\times\cdots\times C_{m}\hookrightarrow A_Z$,	where $m=\dim Z$ and
	\begin{equation*}
		g_*\cF=(\omega_{C_1}\boxtimes\cdots \boxtimes \omega_{C_{m}})\otimes Q,
	\end{equation*}
	for some $Q\in\PAZ$.
	\end{enumerate}
\end{prop}
\begin{proof}
We argue by induction on $\dim Z$.

If $\dim Z=1$, then $Z$ is a smooth projective curve of genus $\geq 1$.
Since $\cF$ is a direct summand of $f_*\omega_V$, we know that $\cF\otimes\omega_Z^{-1}$ is a nef vector bundle on $Z$ (see for instance \cite[Cor.~3.6]{V}).
From $\chi(Z, \cF)=1$, we deduce by Riemann-Roch that $g(Z)\leq 2$ and, if $g(Z)=2$, then $\cF=\omega_Z\otimes Q$ for some $Q\in\PAZ$.

By induction, we may assume the proposition holds for $\dim Z\leq m-1$.

\medskip
\noindent{\bf Step 1.} Suppose that $a_Z(Z)$ is fibred by positive dimensional abelian varieties.

	Then, let $K$ be the maximal abelian subvariety of $A_Z$, such that $a_Z(Z)+ K=a_Z(Z)$.
	We consider the quotient $A_Z\to A_Z/K$ and $Z\stackrel{\kappa}{\to} Z' \to A_Z/K$ the Stein factorization of $Z\to A_Z/K$.
	After birational modifications and abusing notation, we have the following commutative diagram of smooth projective varieties
	\begin{equation*}
	\xymatrix{
	Z\ar[r]^{a_Z}\ar[d]_{\kappa} & A_Z\ar[d]\\
	Z'\ar[r]^{a_{Z'}} & A_{Z}/K,}
	\end{equation*}
	where $A_{Z'}=A_{Z}/K$ and $a_{Z'}$ is indeed the Albanese morphism of $Z'$.

	Since $Z$ is of maximal Albanese dimension, so is the general fiber of $\kappa$.
	Hence, for $Q\in\PAZ$ a general torsion point and $i\geq 1$, $\rk R^i\kappa_*(f_*\omega_V\otimes Q)= 0$, by generic vanishing on the fibres of $\kappa$.
	By Koll\'ar's Theorem \cite[Thm.~2.1]{ko1}, the sheaves $R^i\kappa_*(f_*\omega_V\otimes Q)$ are torsion-free, hence zero for all $i>0$.
	We denote by $\cF'=\kappa_*(\cF\otimes Q)$.
	Since $R^i\kappa_*(\cF\otimes Q)=0$ for all $i\geq 1$, then $\chi(Z', \cF')=\chi(Z,\cF\otimes Q)=\chi(Z, \cF)=1$.

	Let $V'$ be a connected component of the \'etale cover of $V$ induced by $f^*Q$.
	Let $f': V'\to Z'$ be the natural surjective morphism.
	Then, $\cF'$ is clearly a direct summand of $f'_{*}\omega_{V'}$.
	Hence by induction on the dimension of $Z$, we have $q(Z')\leq 2\dim Z'$ and we get $q(Z)\leq 2\dim Z-\dim K<2\dim Z$.

\medskip
\noindent{\bf Step 2.} Suppose that $a_Z(Z)$ is not fibred by positive dimensional abelian varieties.

	We consider the Fourier-Mukai functors
	\begin{align*}
	\R\Phi_\cP : \Db(A_Z) \to \Db(\PAZ), &\quad \R\Phi_\cP(\cdot) : = \R{p_{\PAZ}}_* ( p_{A_Z}^*(\cdot) \otimes \cP),\\
	\R\Psi_\cP : \Db(\PAZ) \to \Db(A_Z), &\quad \R\Psi_\cP(\cdot) : = \R{p_{A_Z}}_* ( p_{\PAZ}^*(\cdot) \otimes \cP),
	\end{align*}
	where $\cP$ is the Poincar\'{e} line bundle on $A_Z\times\PAZ$.

	Since $a_Z(Z)$ is not fibred by positive dimensional abelian varieties, by Lemma \ref{lem:hypo3}{\it (iv)}, $a_{Z*}\cF$ is M-regular.
	Then, by \cite[Thm~2.2 and Cor.~3.2]{PP1}, the following element in $\Db(\PAZ)$,
	\begin{equation*}\widehat{\mathbf{R}\Delta\cF}:=\mathbf{R}\Phi_{\cP}(\mathbf{R}\mathscr{H}om_{A_Z}(a_{Z*}\cF, \cO_{A_Z}))[q(Z)]
	\end{equation*}
	is indeed a torsion-free coherent sheaf in cohomological degree $0$ and of rank $\chi(Z, \cF)=1$.
	Hence, $\widehat{\mathbf{R}\Delta\cF}=M\otimes \cI_W$, where $M$ is a line bundle on $\PAZ$ and $\cI_W$ is the ideal sheaf of $W\subset\PAZ$.

	If $W=\emptyset$, \ie\ $\widehat{\mathbf{R}\Delta\cF}=M$ is a line bundle on $\PAZ$, then $\mathbf{R}\Psi_{\cP}(M)$ is supported on a translate of an abelian subvariety of $A_Z$.
	By Mukai's duality (\cite[Thm.~2.2]{Mu}),
	\begin{equation*}\mathbf{R}\Psi_{\cP}(M)=(-1)_{A_Z}^*\mathbf{R}\mathscr{H}om_{A_Z}(a_{Z*}\cF, \cO_{A_Z}).\end{equation*}
	Thus, $a_{Z*}\cF$ is also supported on a translate of an abelian subvariety of $A_Z$.
	Since $a_{Z*}\cF$ is supported on $a_Z(Z)$ and $a_Z(Z)$ is not fibred by positive dimensional abelian varieties by hypothesis, we get a contradiction.

	We can suppose then that $W$ is not empty.
	That is, $\widehat{\mathbf{R}\Delta\cF}$ is not reflexive so, by \cite[Cor.~3.2]{PP1}, there exists $j>0$
	such that $V^j(\cF)$ has a component $P+\PB$ of codimension $j+1$ in $\PAZ$.

	As usual, we consider $Z\stackrel{\beta}\to Z_B\to B$, the Stein factorization of $Z\to B$.
	After birational modifications and abusing notation, we have the following commutative diagram of smooth projective varieties
	\begin{equation*}
	\xymatrix{
	Z\ar[r]^{a_Z} \ar[d]_{\beta}& A_Z\ar[d]\\
	Z_B\ar[r]^{a_{Z_B}} & B,}
	\end{equation*}
	where $a_{Z_B}:Z_B\to B=A_{Z_B}$ is indeed the Albanese morphism of $Z_B$.

	Note that $\dim Z_B=\dim Z-j$ and $q(Z_B)=\dim B=q(Z)-j-1$.
	As before, for $Q\in \PAZ$ a general torsion point, we have $\chi(Z_B, \beta_*(\cF\otimes Q))=1$.
	Then, considering $\cF_B:=\beta_*(\cF\otimes Q)$, by induction on dimensions, we have $q(Z_B)\leq 2\dim Z_B$.
	Hence, $q(Z)\leq 2\dim Z+1-j\leq 2\dim Z$.

\medskip

We have already proved that $q(Z)\leq 2\dim Z$ in any case, and now we want to study what happens when equality holds.

\medskip
\noindent{\bf Step 3.} The variety $Z$ when equality $q(Z)=2\dim Z$ holds.

	We have seen that in Step 1, we always have an strict inequality.
	In Step 2, equality can hold only when $j=1$ and by induction on dimensions, we may assume $Z_B$ is isomorphic to a product of smooth curves of genus $2$ and $\cF_Q=\beta_*(\cF\otimes Q)=\omega_{Z_B}\otimes Q_B$ for some $Q_B\in\PB$.

	Since $\cF$ is a direct summand of $f_*\omega_V$, $\cF\otimes\omega_Z^{-1}$ is weakly positive (see \cite[Thm.~III]{V1}).
	Thus, $\res{(\cF\otimes \omega_Z^{-1}\otimes Q)}{C}$ is nef, where $C$ is a general fiber of $\beta$.
	Since $\rk \cF_Q=1$, we have $\chi\left(C,\res{(\cF\otimes \omega_Z^{-1}\otimes Q)}{C}\right)=1$, and we conclude that $g(C)=2$ and $\rk \cF=1$.
	Thus, by Proposition \ref{prop:fibertheta}, $Z$ is birational to $C\times Z_B$ and, hence, it is birational to a product of genus $2$ curves.

\medskip
\noindent{\bf Step 4.} The sheaf $\cF$ when equality $q(Z)=2\dim Z$ holds.

	Without loss of generality we may assume that $Z=C_1\times \cdots\times C_m$.
	Denote the natural projections by
	\begin{align*}
		{p}_1: Z&\to C_2\times\cdots\times C_m,\\
		{p_2}: Z&\to C_1\times C_3\times\cdots\times C_m,\text{ and}\\
		{p}_{12}: Z&\to C_3\times\cdots\times C_m.
	\end{align*}
	Fix a general $Q\in\PAZ$ and consider the natural sequence
	\begin{equation}\label{eqn:natseq}
	0\to {p}_{12}^*{p}_{12*}^{}(\cF\otimes Q)\to \begin{smallmatrix}
			{p}_{1}^*{p}_{1*}^{}(\cF\otimes Q)\\
			\oplus\\
			{p}_{2}^*{p}_{2*}^{}(\cF\otimes Q)
	 \end{smallmatrix}\stackrel{\phi}\to \cF\otimes Q.
	\end{equation}
	Note that ${p}_{12}^*{p}_{12*}^{}(\cF\otimes Q)$ lies naturally in the kernel of $\phi$.

	By induction hypothesis and comparing the elements in $\Pic^0(C_k)$ on further push-forwards, we have
	\begin{align*}
		{p}_{1*}(\cF\otimes Q)&=(\omega_{C_2}\otimes Q_2)\boxtimes\left((\omega_{C_3}\boxtimes \cdots\boxtimes\omega_{C_m})\otimes Q_{12}\right),\\
		{p}_{2*}(\cF\otimes Q)&=(\omega_{C_1}\otimes Q_1)\boxtimes\left((\omega_{C_3}\boxtimes\cdots\boxtimes\omega_{C_m})\otimes Q_{12}\right),\text{ and}\\
		{p}_{12*}(\cF\otimes Q)&=(\omega_{C_3}\boxtimes\cdots\boxtimes\omega_{C_m})\otimes Q_{12},
	\end{align*}
	for some $Q_l\in \Pic^0(C_l)$ and $Q_{12}\in \Pic^0(\prod_{k\neq 1,2} C_k)$ that depend on $Q$.
	Now, since $\ker \phi$ has rank $1$, it is easy to see that the sequence \eqref{eqn:natseq} is exact.
	Moreover, if we choose $Q$ appropriately, then
	\begin{equation*}
	0\to \cO_{C_1\times C_2}\to \begin{smallmatrix}
			\cO_{C_1}\boxtimes(\omega_{C_2}\otimes Q_2)\\
			\oplus\\
			(\omega_{C_1}\otimes Q_1)\boxtimes \cO_{C_2}
	 \end{smallmatrix}\to (\omega_{C_1}\otimes Q_1)\boxtimes (\omega_{C_2}\otimes Q_2)\to \cO_x\to 0
	\end{equation*}
	is exact, where $x$ is a point in $C_1\times C_2$ that depends on $Q_1$ and $Q_2$.
	Thus, we have a non-trivial morphism
	\begin{equation*}(\omega_{C_1}\otimes Q_1) \boxtimes (\omega_{C_2}\otimes Q_2)\boxtimes ((\omega_{C_3}\boxtimes \cdots\boxtimes\omega_{C_m})\otimes Q_{12})\longrightarrow \cF\otimes Q.\end{equation*}
	Both sheaves are torsion-free of $\rk=1$ and $\chi=1$, so they are isomorphic (e.g., \cite[Lem.~1.12]{pareschi}).
\end{proof}

Our aim is to apply induction on the dimension to prove Theorem \ref{thm:2n-1main}.
By Theorem \ref{thm:pareschi}, we will be able to reduce to Setting \ref{set:c1}.
But now we are assuming that $a:X\to A$ is the Albanese morphism and we get some extra properties.

\begin{lemm}[Extra properties of Setting \ref{set:c1}] \label{lem:propsett&alb}
	With the same notations as in Setting \ref{set:c1}.
	Assume that $a:X\to A$ is the Albanese morphism $a_X$.

	If $Y=a_X(X)$ is of general type, then the morphism $a_{B}:X_B\stackrel{f_B}\to Y_B\hookrightarrow B$ is the Albanese morphism $a_{X_B}$ of $X_B$.
	Moreover, if $q(X)\geq \dim X +2$, then $Y_B$ is not an abelian subvariety.
\end{lemm}

\begin{proof}
	Note that both $\beta$ and $\pi_B$ are fibrations and $Y_b$ is an ample divisor in a translate of the kernel of $\pi_B$.
	Hence, $a_B$ is the Albanese morphism of $X_B$ and $Y_B\hookrightarrow B$ is the Albanese image of $X_B$.
	If $q(X)\geq \dim X+ 2$, then $\dim B\geq \dim Y_B+1$, so $Y_B$ is not an abelian subvariety of $B$.
\end{proof}

\medskip

Now, we have all the tools to attack Theorem \ref{thm:2n-1main}.

First, we add Condition \eqref{hypo3}, and we get that only case {\it(i)} is possible.
\begin{prop}\label{prop:2n-1A}Let $X$ be a smooth projective variety satisfying \Con\ and $q(X)=2\dim X-1$.
	Then $X$ is birational to a product of theta divisors.
\end{prop}
If $X$ is a product of theta divisors and $q(X)=2\dim X-1$, we clearly need to have $\dim X -1$ theta divisors of dimension $1$ (\ie\ smooth curves of genus $2$) and a theta divisor of dimension $2$ (\ie\ the $2$-dimensional symmetric product of a curve of genus $3$).

\begin{proof}[Proof of Proposition \ref{prop:2n-1A}]
	The surface case was solved independently by Hacon--Pardini \cite{HP-surf} and Pirola \cite{pirola}.
	So, we may assume that $n=\dim X\geq 3$.
	Since $\cO_X\in V^1(\omega_X)$, by Theorem \ref{thm:pareschi}, there exists $\PB$ a non-trivial irreducible component of $V^1(\omega_X)$ passing through $0$.
	Hence, we are in Setting \ref{set:c1} and we use the notation therein, in particular we recall diagram \eqref{eqn:xy1}
	\begin{equation*}
		\xymatrix{
		X_b\ar@{^{(}->}[d]\ar[r]& Y_b \ar@{^{(}->}[d] \\
		X\ar[r]\ar[d]_{\beta} & Y\ar[d]\ar@{^{(}->}[r] & A\ar[d]\\
		X_B\ar[r]& Y_B\ar@{^{(}->}[r] & B.}
	\end{equation*}

	Now, as in the proof of Theorem \ref{thm:sm-in-codim1-sav}, we would like to apply induction by using Proposition \ref{prop:fiber=1}.
	The main difficulty is that $Y_B$ may not be of general type.

	Consider $Q\in \PA_X$ a general point such that $R^i\beta_*(\omega_X\otimes Q)=0$ for all $i>0$ and we
	denote $\cF=\beta_*(\omega_X\otimes Q)$.
	Since $\chi(X,\omega_X)=1$, then $\chi(X_B, \cF)=1$.
	Denote by $i+1=\codim_{\PA_X} \PB$, then $\dim X_B=n-i$ and $q(X_B)=2n-2-i$.
	By Proposition \ref{prop:hpTech} applied to $\beta$ and $\cF$, we have $i\leq 2$.

	If $i=2$, then equality holds in Proposition \ref{prop:hpTech} and $X_B$ is birational to a product of smooth projective curves of genus $2$.
	Hence, we can indeed use Proposition \ref{prop:fiber=1} and we have $\chi(X_b,\omega_{X_b})=1$ and $q(X_b)=\dim X_b+1$.
	Hence, $A_{X_b}$ is isogenous to the kernel of $\pi_B:A_X\to B$ and $Y_b=f_b(X_b)$ is not fibred by positive dimensional abelian subvarieties.
	That is, $X_b$ satisfies the hypothesis of the proposition in dimension $2$.
	Hence by \cite[Thm.~2.2]{HP-surf} or \cite[Thm.~1.1.1]{pirola}, $X_b$ is birational to an irreducible theta divisor of a principally polarized abelian
	threefold.
	Thus, by Proposition \ref{prop:fibertheta}, we have $X$ is birational to $X_B\times X_b$ and we are done.

	If $i=1$, we see that $\dim X_B=n-1$ and $q(X_B)=2n-3$.

	If $Y_B$ is of general type, we can apply again Proposition \ref{prop:fiber=1} to show that the general fiber $X_b$ of $\beta$ is a genus $2$ curve, hence we conclude by induction on dimensions and Proposition \ref{prop:fibertheta}.

	If $Y_B$ is not of general type, then we need to go to a further quotient.
	Let $K$ be the maximal abelian subvariety of $B$ such that $Y_B+K=Y_B$.
	Denote by $i'=i+\dim K$ and denote by $B'$ the quotient $B/K$.
	Let $\pi_{B'}:A_X\to B'$ be the natural projection.
	Then, denote by $Y_{B'}:=\pi_{B'}(Y)$ and let $X\stackrel{\beta'}\to X_{B'} \stackrel{f_{B'}}\to Y_{B'}$ be Stein factorization of $X\to Y_{B'}$.
	Varieties and morphisms described above are compatible with diagram \eqref{eqn:xy1} and fit in the following extended commutative diagram:
	\begin{equation*}
		\xymatrix{
		X\ar[r]^f\ar[d]^{\beta}\ar@/_1.5pc/[dd]_{\beta'} & Y\ar[d]\ar@{^{(}->}[r] & A_X\ar[d]_{\pi_B}\ar@/^1.5pc/[dd]^{\pi_{B'}}\\
		X_B\ar[r]^{f_B}\ar[d]& Y_B\ar[d]\ar@{^{(}->}[r] & B\ar[d]\\
		X_{B'}\ar@/_1.5pc/[rr]_{a_{X_{B'}}}\ar[r]^{f_{B'}} & Y_{B'}\ar@{^{(}->}[r] & {B'}.}
	\end{equation*}
	After birational modifications and abusing notation, we assume that all varieties in the leftmost column are smooth.

	Note that, $Y_{B'}$ is of general type, so also $X_{B'}$ is also of general type.

	Note that $\dim X_{B'}=n-1-\dim K$ and $q(Z)=2n-3-\dim K$.
	Consider $Q\in \PA_X$ a general point such that $R^i\beta'_*(\omega_X\otimes Q)=0$ for all $i>0$ and we
	denote $\cF'=\beta'_*(\omega_X\otimes Q)$.
	Since $\chi(X,\omega_X)=1$, then $\chi(X_{B'}, \cF')=1$.

	By Proposition \ref{prop:hpTech} applied to $\beta'$ and $\cF'$, we have $\dim K=1$ (\ie\ $i'=2$) and $X_{B'}$ is birational to a product of $n-2$ curves of genus $2$.
	Without loss of generality, we can assume that $X_{B'}$ is isomorphic to a product of $n-2$ curves of genus $2$.

	Let $S_{b'}$ be a general fiber of $\beta'$.
	By Proposition \ref{prop:hpTech}, we also get that $\cF'$ has rank $1$, so we have $\chi(S_{b'}, \omega_{S_{b'}})=1$.
	Since $S_{b'}$ is a surface of maximal Albanese dimension and $a_{S_{b'}}(S_{b'})$ is of general type, then either $S_{b'}$ is birational to a symmetric product of curves of genus $3$ (when $q(S_{b'})=3$) or it is birational to a product of curves of genus $2$ (when $q(S_{b'})=4$).

	In the first case, we conclude by Proposition \ref{prop:fibertheta}.
	In the last case, let $C_{b'}$ be the fiber of $X_B\to X_{B'}$ corresponding to $S_{b'}$.
	Note that $C_{b'}$ is of general type, since $X_B$ is of general type.
	Then, the morphism $S_{b'}\to C_{b'}$ needs to be the projection to either $C_1$ or $C_2$.
	In any case, the general fiber of $\beta$ is a curve $C$ of genus $2$.
	By Proposition \ref{prop:fibertheta}, we have that $X$ is birational to $C\times X_B$ and we conclude by induction on the dimension of $X$.
\end{proof}

\begin{rema}
	Note that, a posteriori, we know that, in the previous proof, $X_B$ and $Y_B$ are always birational, so $Y_B$ is of general type.
	In particular, the second case for $i=1$ does not appear.
\end{rema}

\begin{proof}[Proof of Theorem \ref{thm:2n-1main}]
	By Proposition \ref{prop:2n-1A}, we just need to consider varieties not satisfying condition \eqref{hypo3}, \ie\ such that $Y=a_X(X)$ is fibred by positive dimensional abelian subvarieties (see Lemma \ref{lem:hypo3}).

	Consider $K$ to be the largest abelian subvariety of $A_X$ such that $Y+K=Y$ and consider
	the quotient $A_X\to A_X/K$.
	Let $X\stackrel{\kappa}\to Z\to Y/K$ be the Stein factorization of $Z\to Y/K$.
	After birational modifications and abusing notation, we have the following commutative diagram of smooth projective varieties
	\begin{equation}\label{eqn:pic2n-1}
	\xymatrix@!C{
	X\ar[r]\ar[d]_\kappa & Y\ar[d] \ar@{^{(}->}[r] & A_X\ar[d]\\
	Z\ar@/_1.5pc/[rr]_{a_Z}\ar[r]^{f_K}& Y/K\ar@{^{(}->}[r] &A_X/K=A_Z,}
	\end{equation}
	where both $Z$ and $Y/K$ are of general type.

	Note that $\dim Z=n-\dim K$ and $q(Z)=2n-1-\dim K$.
	Consider $Q\in \PA_X$ a general point such that $R^i\kappa_*(\omega_X\otimes Q)=0$ for all $i>0$ and we
	denote $\cF=\kappa_*(\omega_X\otimes Q)$.
	Since $\chi(X,\omega_X)=1$, then $\chi(Z, \cF)=1$.

	By Proposition \ref{prop:hpTech} applied to $\kappa$ and $\cF$, we have $\dim K=1$ and $Z$ is birational to a product of $n-1$ curves of genus $2$.
	Without loss of generality, we can assume that $Z$ is isomorphic to a product of $n-1$ curves of genus $2$.

	The natural morphism $t: X\to Z\times_{A_Z}A_X$ is a generically finite and surjective morphism.
	Thus, $t_*\omega_X=\omega_{Z\times_{A_Z}A_X}\oplus \cQ$ for some torsion-free coherent sheaf $\cQ$ on $Z\times_{A_Z}A_X$.
	We then have $\kappa_*\omega_X=\omega_Z\oplus \cG$, for some torsion-free coherent sheaf $\cG$.
	Note that
	\begin{align*}
	\chi(Z, \cG)&=\chi(Z, \kappa_*\omega_X)-\chi(Z, \omega_Z)\\
	&=\chi(Z, \kappa_*\omega_X)-\chi(Z, R^1\kappa_*\omega_X)\\
	&=\chi(X, \omega_X)=1.
	\end{align*}

	By Proposition \ref{prop:hpTech} applied to $\kappa$ and $\cG$, the sheaf $\cG=\omega_Z\otimes Q$ for some $Q\in \PAZ$.
	In particular, a general fiber of $\kappa$ is a curve $C_1$ of genus $2$.
	We also observe that $Q$ is torsion line bundle by Simpson's theorem (see \cite{sim}) and we denote by $d$ the order of $Q$.

	Let $\pi: \widetilde{Z}\to Z$ be the \'etale cover induced by the cyclic group generated by $Q$.
	We denote the Galois group by $G=\mathbb{Z}/d\mathbb{Z}$.
	We consider the base change of the commutative diagram \eqref{eqn:pic2n-1}:
	\begin{equation*}
	\xymatrix@!C{
	\widetilde{X}\ar[r]\ar[d]^{\widetilde{\kappa}}&X\ar[r]\ar[d]^\kappa& A_X\ar[d]\\
	\widetilde{Z}\ar[r]^{\pi}&Z\ar[r]^{a_Z}& A_X/K.}
	\end{equation*}
	Then
	\begin{align*}
	q(\widetilde{X})=h^{n-1}(\widetilde{X}, \omega_{\widetilde{X}})
	&=h^{n-2}(\widetilde{Z}, R^1\widetilde{\kappa}_*\omega_{\widetilde{X}})+h^{n-1}(\widetilde{Z}, \widetilde{\kappa}_*\omega_{\widetilde{X}})\\
	&=q(\widetilde{Z})+h^{n-1}(\widetilde{Z}, \pi^*\kappa_*\omega_{\widetilde{X}})\\
	&= q(\widetilde{Z})+\sum_{i=0}^{d-1}h^{n-1}(Z, \kappa_*\omega_X\otimes Q^i)\\
	&= q(\widetilde{Z})+\sum_{i=0}^{d-1}h^{n-1}\big(Z,(\omega_Z\oplus (\omega_Z\otimes Q))\otimes Q^i\big)\\
	&= q(\widetilde{Z})+2.
	\end{align*}
	Since a general fiber of $\widetilde{\kappa}$ is still the genus $2$ curve $C_1$,
	by Proposition \ref{prop:fibertheta}, we conclude that $\widetilde{X}$ is birational to a product $C_1\times \widetilde{Z}$.
	Thus, there exists a faithful $G$-action on $C_1$ and $X$ is birational to the quotient by the diagonal action $(C_1\times \widetilde{Z})/G$ (see, e.g., \cite[Lem.~VI.10]{beauv}).
	Since $q(X)=q(Z)+1$, we conclude that $C_1/G$ is an elliptic curve.
	Since $g(C_1)=2$ and $G$ is cyclic, we have $d=2$ and $G=\mathbb{Z}/2\mathbb{Z}$.
\end{proof}

\bigskip

\subsection*{Birational characterization of products of theta divisors}%\label{ssec:qu}
%As we have seen in Proposition \ref{prop:2n-1A}, when $q(X)=2\dim -1$,
It is clear from the previous results that we need the condition \eqref{hypo3} to characterize birationally the products of theta divisors and force $f$ to be birational.
Hence, it is natural ask if this extra condition is enough (the question due to Pareschi, see \cite[Chapter 5]{Ti}).
\begin{qu}[Pareschi]\label{qu?}
Let $X$ be a smooth projective variety.
If $X$ satisfies the $3$ conditions \Con, is $X$ birational to a product of theta divisors?
\end{qu}

If $X$ is birational to a product of irreducible theta divisors, then $X$ clearly satisfies all conditions \Con.
Hence, if the answer is positive this will provide a birational characterization of products of theta divisors.

To our knowledge, until now, the evidences for a positive answer were:
\begin{enumerate}
	\item the surface case $\dim X=2$, starting from the work of Beauville (in the appendix of \cite{appBeau}) and finishing with Hacon--Pardini \cite{HP-surf} and Pirola \cite{pirola};
	\item the maximal irregularity case, \ie\ when $q(X)=2\dim X$, by Hacon--Pardini \cite{HP-g2};
	\item the characterization of theta divisors under some extra assumptions, which started with the work of Hacon \cite[Sect.~3]{Hac-theta}.
	This initial characterization has been refined by Hacon--Pardini \cite[Cor.~4.3]{HP-mAd}, Lazarsfeld--Popa \cite[Prop.~3.8(ii)]{LP},
	Barja--Lahoz--Naranjo--Pareschi \cite[Prop.~3.1]{BLNP}, and Pareschi \cite[Thm.~5.1]{pareschi}.
\end{enumerate}

We have given some new evidences, namely
\begin{enumerate}
	\item[{\it(iv)}] the submaximal irregularity case, \ie\ when $q(X)=2\dim X-1$ (see Proposition \ref{prop:2n-1A});
	\item[{\it(v)}] a further characterization of theta divisors (see Corollary \ref{cor:divisorcase}).
\end{enumerate}

\medskip

We also push forward the characterization of theta divisors of Corollary \ref{cor:divisorcase} to the case of arbitrary codimension of $a_X(X)$ inside $A_X$.
That is, we consider the already mentioned condition \eqref{hypo4}:
\begin{equation*}
\tag{C.4} a_X(X) \text{ is smooth in codimension }1.
\end{equation*}

Then, as a direct corollary of Lemma \ref{lem:sof} and Theorem \ref{thm:sm-in-codim1-sav}, we obtain the following birational characterization of products of theta divisors.
\begin{theo}\label{thm:sm-in-codim1}
Let $X$ be a smooth projective variety.
Then $X$ is birational to a product of theta divisors if, and only if, $X$ satisfies the 4 conditions \Cond.
 \end{theo}

We have already mentioned that Ein--Lazarsfeld prove in \cite{el} that an irreducible theta divisor is normal and has only rational singularities.
Hence, if $X$ is birational to a product of irreducible theta divisors, then $a_X(X)$ is a product of irreducible theta divisors, so it is normal and, in particular, smooth in codimension $1$.

\medskip

It is also possible to give an easier proof of the previous theorem by using induction on the dimension of $X$, instead of induction on the number of components of $\Vz^1(\omega_X,a_X)$.
Indeed, since in this case we work directly with the Albanese variety of $X$, we can avoid using the full machinery of Theorem \ref{thm:sm-in-codim1-sav} and we can directly use Propositions \ref{prop:fibertheta}, \ref{prop:fiber=1}, and Corollary \ref{cor:divisorcase}.

\end{document}